\documentclass[twoside,11pt,reqno]{amsart}

\usepackage[T1]{fontenc}

\DeclareSymbolFont{symbols}{OMS}{cmsy}{m}{n}

\usepackage{fullpage}

\newcommand{\commentout}[1]{}

\usepackage[utf8]{inputenc}
\usepackage[T1]{fontenc}

\usepackage{mathrsfs}
\usepackage{mathtools}
\usepackage{amsfonts}
\usepackage{amsmath,amssymb,amsthm}
\numberwithin{equation}{section} 
\usepackage{comment}

\usepackage{enumerate}
\usepackage{graphicx,tikz}
\usepackage{ifthen}
\usetikzlibrary{arrows,scopes}
\usetikzlibrary{calc,cd}

\usepackage{setspace}
\usepackage[hidelinks]{hyperref} 
\usepackage[alphabetic,initials,nobysame]{amsrefs} 

\newtheorem{thm}{Theorem}[section]

\newtheorem{lem}[thm]{Lemma}
\newtheorem{cor}[thm]{Corollary}
\newtheorem{prop}[thm]{Proposition}
\theoremstyle{definition}
\newtheorem{defi}[thm]{Definition}

\theoremstyle{remark}
\newtheorem{rmk}[thm]{Remark}

\newcommand{\Hil}{\mathcal{H}}

\newcommand{\slot}{{~\cdot~}}

\DeclareMathOperator{\Spin}{Spin}

\DeclareMathOperator{\SO}{SO}
\DeclareMathOperator{\PSL}{PSL}
\DeclareMathOperator{\SU}{SU}

\DeclareMathOperator{\U}{U}

\newcommand{\C}{\mathcal{C}}

\newcommand{\cU}{\mathcal{U}}

\newcommand{\A}{\mathcal{A}}
\newcommand{\B}{\mathcal{B}}
\newcommand{\cI}{\mathcal{I}}

\newcommand{\M}{\mathcal{M}}

\newcommand{\N}{\mathcal{N}}

\newcommand{\RR}{\mathbb{R}}
\newcommand{\TT}{\mathbb{T}}

\newcommand{\CC}{\mathbb{C}}

\newcommand{\ZZ}{\mathbb{Z}}
\newcommand{\NN}{\mathbb{N}}

\DeclareMathOperator{\Diff}{Diff}

\DeclareMathOperator{\Vir}{{Vir}}

\DeclareMathOperator{\Ad}{Ad}

\DeclareMathOperator{\id}{id}

\DeclareMathOperator{\Aut}{Aut}

\newcommand{\lqq}{\lq\lq}
\newcommand{\rqq}{\rq\rq\@\xspace}

\newcommand{\dd}{\,\mathrm{d}}

\usepackage{xspace}
\DeclareRobustCommand{\eg}{e.g.\@\xspace}

\DeclareRobustCommand{\cf}{cf.\@\xspace}
\DeclareRobustCommand{\Cf}{Cf.\@\xspace}
\DeclareRobustCommand{\ie}{i.e.\@\xspace}

\DeclareRobustCommand{\pp}{pp.\@\xspace}

\makeatletter
\DeclareRobustCommand{\etc}{%
    \@ifnextchar{.}%
        {etc}%
        {etc.\@\xspace}%
}
\makeatother

\newcommand{\Cstar}{$C^\ast$\@\xspace}

\def\u1net{{\A_\RR}}

\DeclareMathOperator*{\QuOp}{QuOp}
\DeclareMathOperator*{\UCP}{UCP}
\DeclareMathOperator{\Extr}{Extr}
\def\III{{I\!I\!I}}

\newcommand{\CS}{/\!\!/} 

\renewcommand{\slot}{\,\cdot\,}

\makeatletter
\def\subsection{\@startsection{subsection}{2}%
  \z@{.5\linespacing\@plus.7\linespacing}{.3\linespacing}%
  {\normalfont\bfseries}}
\makeatother

\mathtoolsset{showonlyrefs=true} 

\begin{document}

\date{} 
\dateposted{} 

\title{Quantum Operations on Conformal Nets}

\address{Department of Mathematics, Morton Hall 321, 1 Ohio University, Athens, OH 45701, USA \bigskip}
\author{Marcel Bischoff}
\email{bischoff@ohio.edu}
\address{Dipartimento di Matematica, Universit\`a degli studi di Bari, Via E. Orabona, 4, 70125 Bari, Italy}
\author{Simone Del Vecchio}
\email{simone.delvecchio@uniba.it}
\address{Dipartimento di Matematica, Universit\`a di Roma Tor Vergata, Via della Ricerca Scientifica, 1, I-00133 Roma, Italy}
\author{Luca Giorgetti}
\email{giorgett@mat.uniroma2.it}
\thanks{M.B.\ is supported by NSF DMS grant 1700192/1821162 \emph{Quantum Symmetries and Conformal Nets}. L.G.\ is supported by the European Union's Horizon 2020 research and innovation programme H2020-MSCA-IF-2017 under Grant Agreement 795151 \emph{Beyond Rationality in Algebraic CFT: mathematical structures and models}. We also acknowledge support from the \emph{MIUR Excellence Department Project awarded to the Department of Mathematics, University of Rome Tor Vergata, CUP E83C18000100006}, and from the \emph{University of Rome Tor Vergata funding OAQM,
CUP E83C22001800005}}

\begin{abstract}
On a conformal net $\A$, one can consider collections of unital completely positive maps on each local algebra $\A(I)$, subject to natural compatibility, vacuum preserving and conformal covariance conditions. We call \emph{quantum operations} on $\A$ the subset of extreme such maps. The usual automorphisms of $\A$ (the vacuum preserving invertible unital *-algebra morphisms) are examples of quantum operations, and we show that the fixed point subnet of $\A$ under all quantum operations is the Virasoro net generated by the stress-energy tensor of $\A$. Furthermore, we show that every irreducible conformal subnet $\B\subset\A$ is the fixed points under a subset of quantum operations. 

When $\B\subset\A$ is discrete (or with finite Jones index), we show that the set of quantum operations on $\A$ that leave $\B$ elementwise fixed has naturally the structure of a compact (or finite) hypergroup, thus extending some results of \cite{Bi2016}. Under the same assumptions, we provide a Galois correspondence between intermediate conformal nets and closed subhypergroups. In particular, we show that intermediate conformal nets are in one-to-one correspondence with intermediate subfactors, extending a result of Longo in the finite index/completely rational conformal net setting \cite{Lo2003}.
\end{abstract}

\maketitle
\tableofcontents

\setcounter{tocdepth}{3}


\section{Introduction}

The problems of studying and classifying extensions or subtheories of a given Conformal Field Theory (CFT) are of a different nature, no matter which mathematical (\lqq axiomatic\rqq \ie model independent) formulation one works with. Let us consider for the sake of explaining the difference only \emph{rational} CFTs (those with finitely many inequivalent irreducible positive energy representations, other than the vacuum representation). Extensions can be described using the language and the methods of tensor category theory. While subtheories, to our knowledge and until now, cannot be described tensor categorically in a systematic manner. Nevertheless, in the operator algebraic description of (local and chiral \ie in one spacetime dimension) CFT \cite{Lo2}, \cite{CaKaLoWe2018}, which we shall deal with in this work, the previous statement might sound surprising at first sight. 

A local chiral CFT is formulated (within the more general AQFT setting \cite{Ha}) as a collection of von Neumann algebras $\A(I)$ attached to the proper open intervals $I\subset S^1$ of the unit circle, undergoing a few physically motivated prescriptions (mainly: isotony, locality and conformal covariance of the fields). This description is model independent and based on commonly accepted \lqq first principles\rqq. The collection $\{I\subset S^1 \mapsto \A(I)\}$, denoted by $\A$, is called a local conformal net, or just \emph{conformal net}. One also often specifies the vacuum Hilbert space representation, the projective unitary representation of the group of orientation preserving diffeomorphisms of $S^1$ implementing the conformal covariance, and the vacuum vector \ie the ground state of the conformal Hamiltonian. 

Given a conformal net $\A$, extensions $\A\subset\B$ and subtheories $\B\subset\A$ are both described by \emph{nets of subfactors}, a point of view systematically exploited in \cite{LoRe1995} but present in the literature since the initial works in 3+1 dimensional Minkowski spacetime \cite{DoHaRo1969I}. For every fixed interval $I\subset S^1$, the extension or subtheory is described by a subfactor $\N\subset\M$ (a unital inclusion of von Neumann algebras with trivial center), where the local algebra $\A(I)$ is either $\N$ or $\M$. Assume for the moment that the subfactor has finite Jones index \cite{Jo1983} (roughly speaking: the relative size of $\M$ over $\N$ is finite, although $\M$ and $\N$ are typically infinite-dimensional algebras). Then $\N\subset\M$ can be equally well described by a Q-system \cite{Lo1994} (a unitary Frobenius algebra) in the category of endomorphisms \cite{GiYu2019} (or bimodules) either of $\N$ or of $\M$, in a symmetric fashion. The symmetry is broken however when one wants to describe the whole net of subfactors using Q-systems. 

On the one hand, by \cite{LoRe1995}, finite index extensions $\A\subset\B$ can be characterized by Q-systems in the category of localizable and transportable representations of $\A$  (called DHR representations, after Doplicher--Haag--Roberts \cite{DoHaRo1971}). This description has proven to be extremely powerful, being one of the main tools used to arrive at the classification of conformal nets in the discrete series \cite{KaLo2004}. On the other hand, this method does not adapt to subtheories, as is immediately evident in the case of holomorphic chiral CFTs (those with trivial representation category) which do indeed have non-trivial conformal subnets.

If we no longer restrict ourselves to rational chiral CFTs, or more generally in higher dimensional QFT, infinite index inclusions (extensions and subtheories) may well appear in the analysis of models, for example when one takes theories with compact and non-finite groups of global gauge symmetries into account, see \eg \cite{BuMaTo1988}, \cite{DoRo1990}, \cite{CaCo2001}, \cite{CaCo2001Siena}. It must also be said that finite index extensions have been widely studied in the conformal net (and more generally AQFT) literature, since \cite{LoRe1995}, while a systematic analysis of subtheories is more recent in comparison \cite{Bi2016}. In the possibly infinite index case, conformal net extensions (where the machinery of \cite{DoRo1990} does not apply due to the non-symmetricity of the DHR braiding \cite{FrReSc1989}) have been studied in \cite{DeGi2018}.

In this work, building on the previous analysis of the first named author \cite{Bi2016} in the case of finite index inclusions of completely rational conformal nets, we propose to study subtheories as the fixed points under \emph{quantum operations}. In the first part of the paper, we work in the setting of arbitrary conformal nets (Definition \ref{confnet}) and their subnets (Definition \ref{confsubnetdef}). In the second part, we restrict ourselves to discrete conformal subnets (Definition \ref{discretesubnets}), which also cover the case of finite index subnets (Definition \ref{finindexsubnets}). In this second part, we show that our previous analysis of local discrete subfactors \cite{BiDeGi2021}, \cite{BiDeGi2022} applies to conformal subnets as well. In more detail, a quantum operation on $\A$ (Definition \ref{QuOpA}) is a collection of unital completely positive maps $\A(I) \to \A(I)$, indexed by $I\subset S^1$, that are compatible with the inclusions of local algebras $\A(I) \subset \A(J)$ for $I \subset J$, vacuum preserving and conformally covariant in a natural sense, and \emph{extreme} in the sense of convex sets among all unital completely positive maps on $\A$. We denote by $\QuOp(\A)$ the set of quantum operations. The terminology is inspired by quantum information theory \cite{OhPe1993}, \cite{WildeBook}, where unital completely positive maps (typically between finite-dimensional \Cstar-algebras) describe quantum channels. We show that the automorphisms of $\A$ (the most commonly considered type of symmetry transformation in AQFT, Definition \ref{AutA}) are quantum operations (Proposition \ref{AutinQuOp}), and that the set of all unital completely positive maps on $\A$ is compact and Hausdorff in the pointwise ultraweak operator topology over every interval $I\subset S^1$ (Theorem \ref{UCPcompact}). 

Our first main result (Theorem \ref{fixedpointsVir}) states that the fixed point subnet of $\A$ under all quantum operations is the minimal and canonical Virasoro subnet (generated by the stress-energy tensor). Consequently, every subnet $\B$ of $\A$ that contains the Virasoro subnet of $\A$, or equivalently such that $\B\subset\A$ is irreducible (Definition \ref{irredinclusion}) by Proposition \ref{irrediffintermediate}, is the fixed points under a subset of quantum operations. 
More generally, given an irreducible conformal inclusion $\B\subset\A$, denote by $\QuOp(\A|\B)$ the subset of quantum operations on $\A$ that leave $\B$ elementwise fixed. If $\B\subset\A$ is discrete, we show that $\QuOp(\A|\B)$ is closed in the compact Hausdorff space of all unital completely positive maps on $\A$ (hence compact and Hausdorff with the induced topology) and that it naturally forms a hypergroup (Theorem \ref{QuOpHyper}). Hypergroups (Definition \ref{abstractcpthyp}) are a classical generalization of group and they are well-suited for abstract harmonic analysis. An abstract convolution replaces the group operation, an involution replaces the group inversion, there is an identity element and a Haar measure (finite in the case of compact hypergroups). The key point in the proof of Theorem \ref{QuOpHyper}, besides applying our previous results on local discrete subfactors \cite{BiDeGi2021}, \cite{BiDeGi2022}, is to show that every $\B(I)$-fixing unital completely positive map $\A(I) \to \A(I)$, for fixed $I\subset S^1$, with no additional assumption, can be extended to a compatible, covariant and vacuum preserving family of $\B$-fixing maps on the whole net $\A\to\A$ (Theorem \ref{discretecase}). We don't know whether the same statement is true for arbitrary irreducible conformal inclusions (not assuming discreteness).

Our second main result (Theorem \ref{Galois}), assuming that $\B\subset\A$ is discrete, provides a one-to-one correspondence between the closed subhypergroups of $\QuOp(\A|\B)$ and the conformal subnets of $\A$ that contain $\B$. In particular, we show that the latter are in one-to-one correspondence with the intermediate von Neumann algebras $\B(I) \subset\N\subset\A(I)$, for fixed $I\subset S^1$, a result due to Longo \cite{Lo2003} in the case of finite index inclusions of completely rational conformal nets.

\section{Conformal nets and subnets}\label{prelim}

Let $\mathrm{PSL}(2,\RR) := \mathrm{SL}(2,\RR)/\{\pm 1\}$ and $\mathcal{I}$ be the set of non-empty, non-dense, open intervals $I$ of the unit circle $S^{1}$. $\mathrm{PSL}(2,\RR)$ acts on $S^{1}$ by M\"obius transformations, see, \eg, \cite[Chapter 1]{Lo2}, \cite[Appendix I]{GaFr1993}.
Denote by $I' := (S^{1}\setminus I)^{\circ}$ the interior of the complement of the interval $I\in\mathcal{I}$.
Denote also by $\B(\Hil)$ the algebra of bounded operators on $\Hil$ and by $\cU(\Hil)$ the unitary subgroup.

\begin{defi}
A {\bf M\"obius covariant net} on $S^{1}$ is a triple $(\A, U, \Omega)$ consisting of a family of von Neumann algebras $\A=\left\{\A(I) \subset\B(\Hil): I\in\mathcal{I}\right\}$ acting on a common complex separable Hilbert space $\Hil$, a strongly continuous unitary representation $U : \mathrm{PSL}(2,\RR) \to \cU(\Hil)$ and a unit vector $\Omega \in \Hil$, satisfying the following properties:
\begin{enumerate}
\item[(i)] \textbf{Isotony}: $\A(I_{1})\subset\A(I_{2})$, if $I_{1}\subset I_{2}$, $I_{1},I_{2}\in \mathcal{I}$.
\item[(ii)] \textbf{Locality}: $\A(I_{1})\subset\A(I_{2})'$, if $I_{1}\cap I_{2}=\emptyset$, $I_{1},I_{2}\in \mathcal{I}$.
\item[(iii)] \textbf{M\"obius covariance}: for every $I\in\mathcal{I}$, $g\in \mathrm{PSL}(2,\RR)$,
\begin{align}
U(g)\A(I)U(g)^{-1}=\A(gI).
\end{align}
\item[(iv)] \textbf{Positivity of energy}: $U$ has positive energy. Namely, the conformal Hamiltonian (the generator of the one-parameter rotation subgroup of $\mathrm{PSL}(2,\RR)$) has non-negative spectrum.
\item[(v)] \textbf{Vacuum vector}: $\Omega$ is the unique vector (up to a phase) with the property
$U(g)\Omega=\Omega$ for every $g\in \mathrm{PSL}(2,\RR)$, and vectors of the form $x\Omega$, $x\in\bigvee_{I\in\mathcal{I}}\A(I)$, are dense in $\Hil$.
\end{enumerate}

Here $\bigvee_{I\in\mathcal{I}}\A(I)$ denotes the von Neumann algebra generated in $\B(\Hil)$ by the $\A(I)$, $I\in\mathcal{I}$, and $\A(I)'$ denotes the commutant of $\A(I)$ in $\B(\Hil)$, namely $\A(I)' := \{x\in\B(\Hil): xy = yx, y\in\A(I)\}$. The $\A(I)$ are referred to as the \textbf{local algebras} of $\A$ and $\Hil$ as the \textbf{vacuum Hilbert space} of $\A$.
\end{defi}

With these assumptions, the following properties automatically hold. See \cite[Theorem 2.3]{BrGuLo1993}, \cite[Theorem 2.19]{GaFr1993}, \cite[Section 1]{GuLo1996}, \cite[Section 3]{FrJr1996}, \cite[Chapter 3]{CaKaLoWe2018}. Let $I\in\cI$, then

\begin{enumerate}
\item \textbf{Reeh--Schlieder theorem}: $\Omega$ is cyclic and separating for $\A(I)$. Namely, vectors of the form $x\Omega$, $x\in\A(I)$, are dense in $\Hil$, and $x\Omega = 0$ implies $x = 0$.
\item \textbf{Bisognano--Wichmann theorem}: Denote by $\Delta_I$ and $J_I$ respectively the Tomita--Takesaki modular operator and antiunitary conjugation (for whose definition we refer to \cite{BrRo1}) associated with $\A(I)$ and $\Omega$. Denote by $\delta_I(t)$, $t\in\RR$, the one-parameter dilation subgroup of $\PSL(2,\RR)$ associated with $I$ (the special conformal transformations that preserve $I$). Then $\Delta_I^{it} = U(\delta_I(2\pi t))$ for every $t\in\RR$, and $J_I$ acts as the reflection mapping $I$ to $I'$.
\item \textbf{Haag duality}: $\A(I)'=\A(I')$.
\item \textbf{Factoriality}: As a consequence of the uniqueness of the vacuum vector, $\A(I)$ is a factor, necessarily of type $\III_1$ in Connes' classification \cite{Co1973}. Equivalently, $\bigvee_{I\in\mathcal{I}}\A(I) = \B(\Hil)$, \ie, $\left(\bigvee_{I\in\mathcal{I}}\A(I)\right)' = \CC 1$.
\end{enumerate}

The Bisognano--Wichmann theorem implies in particular that the M\"obius covariance (the unitary representation $U$ of $\PSL(2,\RR)$) can be reconstructed from the datum of the local algebras and the vacuum vector, see \cite{GuLoWi1998}. We shall assume throughout this paper the stronger covariance property under diffeomorphisms:

\begin{defi}\label{confnet}
Let $\Diff_+(S^1)$ be the group of orientation preserving diffeomorphisms of $S^1$. 
By a \textbf{conformal net} (or \textbf{diffeomorphism covariant net}) on $S^1$ we shall mean a M\"obius covariant net $(\A,U,\Omega)$ which satisfies in addition:
\begin{enumerate}
\item[(vi)] The representation $U$ of $\PSL(2,\RR)$ extends to a strongly continuous projective unitary representation of $\Diff_+(S^1)$, again denoted by $U$, such that for every $I\in\mathcal{I}$:
\begin{align}
U(\gamma)\A(I)U(\gamma)^{-1}&=\A(\gamma I), \quad\gamma\in\Diff_+(S^1),\\
U(\gamma)xU(\gamma)^{-1}&=x,\quad x\in \A(I), \gamma\in\Diff_+(I^\prime),
\end{align}
where $\Diff_+(I^\prime)$ denotes the subgroup of orientation preserving diffeomorphisms of $S^1$ that are \emph{localized} in $I'$, namely $\gamma\in\Diff_+(S^1)$ such that $\gamma(z)=z$ for all $z\in I$.
\end{enumerate}
\end{defi}

Note that the unitaries $U(\gamma)$ are only defined up to a phase. Moreover, by the second equation above and by Haag duality on $\Hil$, it follows that $U(\gamma) \in \A(I')$ if $\gamma\in\Diff_+(I^\prime)$.

\begin{defi}\label{confsubnetdef}
A \textbf{conformal subnet} of a conformal net $(\A,U,\Omega)$ is a family $\B= \{\B(I) : I\in\mathcal{I}\}$ of non-trivial von Neumann algebras acting on $\Hil$ such that:
\begin{enumerate}
\item[(i)] $\B(I) \subset \A(I)$ for every $ I \in \mathcal{I}$.
\item[(ii)]  $U(g)\B(I)U(g)^{-1} = \B(g I)\,$ for every $I \in \mathcal{I}, g \in \PSL(2,\RR)$.
\item[(iii)] $\B(I_1)\subset \B(I_2)$ for every $I_1,I_2\in\mathcal{I}$ with $I_1\subset I_2$.
\end{enumerate}
\end{defi}

\begin{rmk}\label{confsubnetU}
By \cite[Theorem 6.2.29]{WeiPhD}, \cf \cite[Section 3.4]{CaKaLoWe2018}, a conformal subnet $\B\subset\A$ fulfills also diffeomorphism covariance:
\begin{align}
U(\gamma)\B(I)U(\gamma)^{-1}=\B(\gamma I), \quad I\in\cI, \gamma\in\Diff_+(S^1).
\end{align}
\end{rmk}

We call $\B\subset\A$ a \textbf{conformal inclusion} (or sometimes conformal subnet, when we want to stress the role of $\B$, with abuse of terminology). Note that $\B$ restricted to the Hilbert subspace $\Hil_\B\subset\Hil$ obtained as the closure of $\vee_{I\in\mathcal{I}}\B(I)\Omega$ is a conformal net. Indeed, $\B$ is clearly M\"obius covariant with the same $U$ of $\A$ restricted to $\Hil_\B$. Moreover, by \cite[Theorem 6.2.31]{WeiPhD}, it also admits a strongly continuous projective unitary representation of $\Diff_+(S^1)$ on $\Hil_\B$, that we denote by $U_\B$ for later reference, extending the restriction of $U$ to $\PSL(2,\RR)$ and fulfilling the conditions in Definition \ref{confnet}.

\begin{rmk}\label{stdcondexpandJonesproj}
Let $\B\subset\A$ be a conformal inclusion. By the Bisognano--Wichmann theorem and by Takesaki's theorem \cite{Ta1972}, for every $I\in\cI$, there is a normal faithful conditional expectation on the subfactor $\B(I)\subset\A(I)$, denoted by $E_I:\A(I) \to \B(I)\subset\A(I)$, uniquely determined by the vacuum state preserving condition $\omega_I \circ E_I = \omega_I$. We refer to \cite{St1997} for a concise overview of conditional expectations on von Neumann algebras.
Here $\omega_I := (\Omega, \slot \Omega)$ is the vacuum state of $\A$ restricted to $\A(I)$. The conditional expectation is implemented by the Jones projection $e_I := [\B(I)\Omega]$, \cite{Jo1983}, via the formula $e_I x e_I = E_I(x) e_I$ for every $x\in\A(I)$. Also, $x\in\A(I)$ belongs to $\B(I)$ if and only if $e_I x = x e_I$. By using $\B(I) = E_I(\A(I))$, the Jones projection is equivalently defined as 
\begin{align}\label{Jonesproj}
e_I x \Omega := E_I(x) \Omega, \quad x\in\A(I).
\end{align}

By the Reeh--Schlieder theorem for $\B$ on $\Hil_\B$, the Jones projection is independent of $I\in\cI$ and it coincides with the orthogonal projection onto the Hilbert subspace $\Hil_\B \subset \Hil$. We write $e_\B := e_I$. Consequently, the collection of conditional expectations $E_\B := \{E_I : I\in\cI\}$ is compatible in the sense that $E_I(x) = E_J(x)$ if $I\subset J$, $x\in\A(I)$. Hence $E_\B$ is a \textbf{standard conditional expectation} of $\A$ onto $\B$ in the terminology of \cite[Definition 3.1]{LoRe1995}.
\end{rmk}

\begin{defi}\label{irredinclusion}
We call an inclusion $\B\subset\A$ \textbf{irreducible} if $\B(I)' \cap \A(I) = \CC 1$ for some, hence for all, $I\in\cI$, where the commutant is taken in $\B(\Hil)$. 
\end{defi}

By irreducible conformal subnet we shall mean a conformal subnet in the sense of Definition \ref{confsubnetdef} such that the inclusion is irreducible in the sense of Definition \ref{irredinclusion}. Note that a conformal subnet $\B\subset \A$ is irreducible if and only if it is a \emph{full subsystem} in the sense of \cite[Section 3]{Ca2004} (namely: the \emph{coset net} of $\B$ into $\A$ is trivial) thanks to a result of \cite{Koe2004}.

\begin{rmk}\label{irreduniqueE}
If $\B\subset\A$ is conformal and irreducible, then $E_\B$ is the unique among normal faithful (a priori not necessarily vacuum preserving) conditional expectations of $\A$ onto $\B$, \eg, by \cite[Theorem 5.3]{CoDe1975}.
\end{rmk}

Let $U_{(c,0)}$ be the irreducible strongly continuous projective unitary and positive energy representation of $\Diff_+(S^1)$ with central charge $c$ and lowest weight zero. See \cite{FrQiSh1985}, \cite{GoKeOl1986}, \cite{KaRaBombay}.

\begin{defi}
Let $\Vir_c$ be the \textbf{Virasoro net} with central charge $c$ associated with $U_{(c,0)}$:
\begin{align}
\left(\Vir_c(I):=\{U_{(c,0)}(\gamma): \gamma\in\Diff_+(I)\}^{\prime\prime}, U_{(c,0)},\Omega \right)
\end{align}
for every $I\in\cI$, where $\Omega$ is the lowest weight vector of $U_{(c,0)}$.
\end{defi}

Any conformal net $(\A,U,\Omega)$ contains a copy of $\Vir_c$, for some $c$, as a conformal subnet: 
\begin{align}
\left(\Vir_\A(I) := \{U(\gamma): \gamma\in\Diff_+(I)\}^{\prime\prime}, U,\Omega\right)
\end{align}
for every $I\in\cI$, \cf \cite[Remark 3.8]{Ca2004}. By \cite[Proposition 3.7 (a)]{Ca2004}, the conformal inclusion $\Vir_\A \subset \A$ is automatically irreducible, \ie, $\Vir_\A(I)' \cap \A(I) = \CC 1$ for every $I\in\cI$. Moreover, the Virasoro net is minimal, in the sense it does not contain any non-trivial conformal subnet \cite{Ca1998}. For later use, we denote by $E_{\Vir}$ the standard conditional expectation of $\A$ onto $\Vir_\A$, which is unique by Remark \ref{irreduniqueE}, and by $e_{\Vir}$ the associated Jones projection as in \eqref{Jonesproj}. 

We recall the following:

\begin{prop}\label{irrediffintermediate}
Let $(\A,U,\Omega)$ be a conformal net and $\B\subset\A$ a conformal subnet. Then $\B$ contains $\Vir_\A$, namely it is intermediate $\Vir_\A \subset \B \subset\A$, if and only if the inclusion $\B\subset\A$ is irreducible.

Moreover, the condition $\Vir_\A \subset \B \subset\A$ is equivalent to $\Vir_\B = \Vir_\A$.
\end{prop}

\begin{proof}
The fact that $\Vir_\A \subset \B$ implies the irreducibility of $\B\subset\A$ is immediate from the previous discussion. Indeed, $\B(I)' \cap \A(I) \subset \Vir_\A(I)' \cap \A(I) = \CC 1$. 
The converse implication can be proven as follows. By Haag duality, $U(\gamma)\in\A(I)$ for every $\gamma\in\Diff_+(I)$, $I\in\cI$. By Remark \ref{confsubnetU}, $\Ad U(\gamma) = U(\gamma) \slot U(\gamma)^{-1}$ is an automorphism of $\B(I)$. Let $E_\B = \{E_I : I\in\cI\}$ be the standard conditional expectation of $\A$ onto $\B$. Then $U(\gamma)^{-1} E_I(U(\gamma)) \in\A(I) \cap \B(I)'$, as one can check using the $\B(I)$-bimodularity of $E_I$, hence it holds $E_I(U(\gamma)) = \lambda U(\gamma)$ for some $\lambda\in\CC$, by the irreducibility assumption. Moreover, either $E_I(U(\gamma)) = U(\gamma)$, if $U(\gamma) \in \B(I)$, or $E_I(U(\gamma)) = 0$, otherwise. We can exclude the second case as follows.
Let $f\in C^\infty (S^1)$ be a smooth real valued function on $S^1$ and let $T(f)$ be the stress-energy tensor associated with $\Vir_\A$, see, \eg, \cite[Section 3.2, 3.3.B]{FeHo2005} and \cite[Section 2]{CDIT2021} for a short review. Note that $\Vir_\A(I)$ is generated as a von Neumann algebra by elements of the form $e^{itT(f)}$, $t\in\RR$, with $f\in C^\infty (S^1)$ having support inside $I$, see, \eg,  \cite[\pp 267--268]{Ca2004}. Note also that $E$ is (ultra)weakly/strongly operator continuous, being completely positive and normal, see, \eg, \cite[Proposition III.2.2.2]{BlaBook}, and $\B(I)$-bimodular by definition.
Hence it suffices to show that $E(e^{it T(f)}) = e^{itT(f)}$ for every $t \in\RR$ and for every $f \in C^\infty(S^1)$ with support in $I$. The subset of $\RR$ given by $A_f := \{t \in\RR: E(e^{it T(f)}) = e^{itT(f)}\}$ is closed. But $A_f$ is also open, since, as argued above, its complement is ${A_f}^c := \{t \in\RR : E(e^{itT(f)}) = 0\}$, which is again closed. Since $A_f$ is non-empty, as it contains the point $t = 0$, by connectedness we conclude that $A_f = \RR$.
Thus $U(\gamma) \in \B(I)$ for every $\gamma\in\Diff_+(I)$, and we conclude that $\Vir_\A\subset\B$. The statement just proven that the irreducibility of $\B\subset\A$ implies $\Vir_\A \subset \B$ also follows from \cite[Theorem 6.2.31]{WeiPhD}, see \cite[Corollary 6.3.7]{WeiPhD}.

For the second statement, if $\Vir_\B = \Vir_\A$ then $\B$ is intermediate. Conversely, if $\Vir_\A \subset \B \subset\A$, then $U$ preserves the Hilbert subspace $\Hil_\B\subset\Hil$. Hence it must coincide with the $U_\B$ introduced in Remark \ref{confsubnetU} by the uniqueness result that we shall recall in Remark \ref{uniqueUofDiff}.  
\end{proof}

\begin{rmk}
Conformal inclusions with \emph{finite Jones index} are automatically irreducible \cite[Lemma 14]{Lo2003}, \cf \cite[Corollarly 3.6]{BcEv1998-I}, \cite[Corollary 2.7]{DLR2001}. We shall recall in the next section the definition of Jones index for conformal inclusions (Definition \ref{finindexsubnets}). 
Hence Proposition \ref{irrediffintermediate} recovers \cite[Proposition 6.2]{KaLo2004} in the finite index case.
\end{rmk}

Typical examples of irreducible conformal subnets come from finite or compact group actions on conformal nets \cite{Xu2000orbi}, \cite[Section 2]{Xu2005}, \cite[Section 3]{Ca2004}, inspired by the study of global gauge group symmetries in 3+1 dimensional Minkowski spacetime \cite[Section II]{DoHaRo1971}, \cite{DoRo1990}. 

\begin{defi}\label{groupaction}
A finite or compact group $G$ \textbf{acts properly} on $(\A,U,\Omega)$ if there is a faithful strongly continuous unitary representation $V : G \to \cU(\Hil)$ such that:
\begin{enumerate}
\item[(i)] $V(g) \A(I) V(g)^{-1} = \A(I)$ for every $I\in\cI$, $g\in G$.
\item[(ii)] $V(g) \Omega = \Omega$ for every $g\in G$.
\item[(iii)] $U(h) V(g) = V(g) U(h)$ for every $g\in G$, $h\in\PSL(2,\RR)$.
\end{enumerate}
\end{defi}

\begin{rmk}
In other words, $\Ad V(g) = V(g) \slot V(g)^{-1}$ is a vacuum preserving automorphism $\A(I) \to \A(I)$ for every $I\in\cI$, commuting with the M\"obius action. The condition (iii) above follows from (i) and (ii), \cite[Appendix II]{GaFr1993}, \cite[Section 3]{Xu2000orbi}. We shall provide a proof of this statement in a more general context in Section \ref{discrete}.
\end{rmk}

By setting $\A^G(I) := \A(I) \cap V(G)'$ one obtains a conformal subnet of $(\A,U,\Omega)$, called the \textbf{orbifold} (or \textbf{fixed point subnet}) of $\A$ with respect to $G$. The subspace $\Hil_{\A^G} \subset \Hil$ defined as the closure of $\vee_{I\in\mathcal{I}}\A^G(I)\Omega$ coincides with the subspace of $V$-invariant vectors in $\Hil$, and the inclusion $\A^G \subset \A$ is automatically irreducible \cite[Proposition 2.1]{Xu2001siena}, \cite[Proposition 2.1]{Ca1999}. 
Proposition \ref{irrediffintermediate} implies that $\Vir_\A \subset \A^G$, or equivalently that $U(\gamma) V(g) = V(g) U(\gamma)$ for every $g\in G$, $\gamma\in\Diff_+(I)$. Thus also for every $\gamma\in\Diff_+(S^1)$, as $\Diff_+(S^1)$ is algebraically simple, see, \eg, \cite{Mi1984}, hence generated by localized diffeomorphisms. The commutation relation $U(\gamma) V(g) = V(g) U(\gamma)$ is unambiguously written $U(\gamma) V(g) U(\gamma)^{-1} = V(g)$, as the unitaries $U(\gamma)$ are only defined up to a phase.

\begin{defi}\label{AutA}
Let $(\A,U,\Omega)$ be a conformal net. Let $\Aut(\A)$ be the set of \textbf{automorphisms} of $\A$. Namely, $\alpha\in\Aut(\A)$ if it is of the form $\alpha = \Ad V$ for some $V\in\cU(\Hil)$ such that $V \A(I) V^{-1} = \A(I)$ for every $I\in\cI$, $V \Omega = \Omega$ and $U(\gamma) V U(\gamma)^{-1} = V$ for every $\gamma\in\Diff_+(S^1)$.
\end{defi}

\begin{rmk}\label{uniqueUofDiff}
Under a further regularity assumption on the conformal net, it is known by \cite[Corollary 5.8]{CaWe2005} that the first two conditions on $V$ in the previous definition imply the third. This is a consequence of the uniqueness, when it exists, of the extension of $U$ from $\PSL(2,\RR)$ to $\Diff_+(S^1)$. Both statements are true in full generality by \cite[Theorem 6.1.9]{WeiPhD}, \cf \cite[Theorem 6.10]{CaKaLoWe2018}.
\end{rmk}

\section{Representations of conformal nets}\label{RepnConfNets}

In this section, we recall the definition of representation, following Doplicher--Haag--Roberts \cite{DoHaRo1971}, \cite{DoHaRo1974} in the special case of conformal nets. See \cite{FrReSc1989}, \cite{FrReSc1992}, \cite{GuLo1996}, \cite{DVIT20}.

By a representation $\pi$ of a von Neumann algebra $\M$ on a Hilbert space $\Hil_\pi$, we mean a normal unital *-algebra morphism $\pi:\M\to\B(\Hil_\pi)$, see, \eg, \cite[Chapter III.3]{Ta1}. In our case at hand where $\M$ is an infinite factor realized on a separable Hilbert space $\Hil$, by \cite[Theorem V.5.1]{Ta1}, if $\Hil_\pi$ is also separable then $\pi$ is automatically normal.

\begin{defi}
Let $(\A,U,\Omega)$ be a conformal net.  A \textbf{representation} $\pi$ of $(\A,U,\Omega)$ is a collection: 
\begin{align}
\pi=\{ \pi_I : I\in\mathcal{I}\},
\end{align}
where each $\pi_I$ is a representations of $\A(I)$ on a common Hilbert space $\Hil_\pi$, fulfilling the compatibility condition $\pi_{I_2}\restriction_{\A(I_1)}=\pi_{I_1}$ for every $I_1,I_2\in \mathcal I$ with $I_1\subset I_2$. 

A representation $\pi$ is called \textbf{irreducible} if $\left(\bigvee_{I\in\mathcal{I}}\pi_I(\A(I))\right)' = \CC 1_{\Hil_\pi}$.
\end{defi}

Two representations $\pi$ and $\sigma$ of $\A$ are unitarily equivalent if there is a unitary $V:\Hil_\pi \to \Hil_\sigma$ intertwining $\pi$ and $\sigma$, \ie, such that $V\pi_I(x) = \sigma_I(x) V$ for every $I\in\cI$, $x\in\A(I)$.
Due to the type $\III$ property of local algebras, by \cite[Theorem V.3.2]{Ta1}, every separable representation $\pi$ is locally unitarily equivalent to the defining \textbf{vacuum representation} $\pi_0$ on $\Hil$. Namely, for every $I\in\cI$ there is a unitary $V_I : \Hil_\pi \to \Hil$ such that $V_I \pi_I(x) = x V_I$ for every $x\in\A(I)$. In particular, every $\pi_{I_1}$ is unitarily equivalent to every other $\pi_{I_2}$ by means of a unitary intertwiner depending on $I_1$ and $I_2$. The morphism $\rho_{I'} := V_I \pi_{I'}(\slot) V_I^{-1} : \A(I') \to \B(\Hil)$ can be shown to be an \textbf{endomorphism} of $\A(I')$ by Haag duality. The morphism $\rho_{I}(x) := V_I \pi_{I}(x) V_I^{-1} : \A(I) \to \B(\Hil)$ instead is the trivial endomorphism, $\rho_{I}(x) = x$, for every $x\in\A(I)$. The collection $\rho := \{\rho_J : J\in\mathcal{I}\}$ defined by $\rho_{J} := V_I \pi_{J}(\slot) V_I^{-1} : \A(J) \to \B(\Hil)$ is a representation of $\A$ on $\Hil$, unitarily equivalent to $\pi$. Due to the trivialization property $\rho_{I}(x) = x$, the representation $\rho$ is called \textbf{localized} in $I^\prime$.

The category of representations of $\A$ and intertwiners is linear, unitary and \Cstar, see \cite{EGNO15}, \cite{Mue10tour}. A crucial consequence of Haag duality is that intertwiners between localized representations belong to local algebras. This allows to endow the representation category with a \textbf{tensor} (or \textbf{monoidal}) product operation, locally defined on objects by the composition of endomorphisms. In particular, one can consider the (intrinsic) tensor \Cstar-categorical notion of \textbf{dimension} for representations (not the Hilbert space dimension of $\Hil_\pi$), see \cite{LoRo1997}, \cite{GiLo2019}. The intrinsic dimension, denoted by $d(\pi)$, preserves direct sums and tensor products of representations, $d(\pi) \in [1,\infty]$ and $d(\pi_0) = 1$. Moreover, the locality assumption on $\A$ endows the representation category with an additional unitary \textbf{braided} structure given by the DHR braiding, see \cite{BiKaLoRe2014-2}, \cite{GiRe2018}.

Let now $\B\subset \A$ be a conformal subnet as in Definition \ref{confsubnetdef}. The definition given below coincides with the one of \emph{compact type} inclusion of conformal nets considered in \cite[Section 3]{Ca2004}. It should be compared with \cite[Section 5]{LoRe1995} and with the seemingly weaker notion of discreteness considered in \cite[Section 6]{DeGi2018}\footnote{An inclusion of (not necessarily chiral, nor conformal) nets $\B\subset\A$ is called discrete in \cite[Definition 6.7]{DeGi2018} if every subfactor $\B(I)\subset\A(I)$, $I\in\cI$, is discrete in the sense of \cite[Definition 3.7]{IzLoPo1998}.}.

\begin{defi}\label{discretesubnets}
Denote by $\iota$ the defining representation of $\B$ on $\Hil$, \ie, the restriction to $\B$ of the vacuum representation of $\A$.
An irreducible conformal inclusion $\B\subset\A$ is said to be \textbf{discrete} if $\iota$ decomposes as a countable direct sum of irreducible representations of $\B$, all with finite dimension. 
\end{defi}

If $\B\subset\A$ is irreducible and discrete, then every \emph{subfactor} $\B(I)\subset\A(I)$, $I\in\cI$, is irreducible by the very Definition \ref{irredinclusion}, \emph{discrete} in the sense of Izumi--Longo--Popa \cite[Definition 3.7]{IzLoPo1998}, and \emph{(braided) local} in the sense of \cite[Definition 2.16]{BiDeGi2021} with respect to the DHR braiding, and type $\III$. See \cite[Proposition 9.20]{BiDeGi2021}. 
See also \cite[Section 5]{DeGi2018}, \cite[Section 2]{BiDeGi2021} for some results and further references on discrete and (braided) local subfactors, and \cite[Section 1 and 2.3]{BiDeGi2022} for a concise overview. In the finite index case (see below), locality for subfactors is also known as \emph{chiral locality} \cite{BcEv1998-I}, \cite{BcEvKa1999} or \emph{commutativity} of the associated Q-system \cite{BiKaLoRe2014-2}.

If, in addition, $\iota$ has \emph{finitely many} irreducible direct summands, then every $\B(I)\subset\A(I)$ has \emph{finite Jones/minimal index} \cite{Jo1983}, \cite{Lo1989}, see \cite[Section 1]{Gi2022} for an overview of the various notions of index for subfactors and their relations. The index value $[\A(I):\B(I)]$ is independent of $I\in\cI$ by conformal covariance and we denote it by $[\A:\B]$.

\begin{defi}\label{finindexsubnets}
A (necessarily irreducible) conformal inclusion $\B\subset\A$ is said to have \textbf{finite index} if $\iota$ has finitely many irreducible direct summands, all with finite dimension. 
\end{defi}

\section{Quantum operations on conformal nets}\label{QuOponConfNets}

Recall the definition of unital completely positive map on a von Neumann algebra $\M$, see, \eg, \cite[Chapter IV.3]{Ta1}. A linear map $\phi: \M\to\M$ is called unital if $\phi(1)=1$. It is called positive if $x\geq 0$, $x\in\M$, or equivalently $x = y^*y$, $y\in\M$, implies $\phi(x) \geq 0$. Lastly, it is called completely positive if the ampliation $\phi \otimes \id_n\colon \M\otimes M_n(\CC)\to\M\otimes M_n(\CC)$ is positive for every $n\in\NN$, where $M_n(\CC)$ are the complex $n\times n$ matrices. Special examples of unital completely positive maps are given by conditional expectations and automorphisms. We refer to \cite{Pa2002} for more background.

\begin{defi}\label{UCPA}
Let $(\A,U,\Omega)$ be a conformal net. Let $\UCP(\A)$ be the set of \textbf{unital completely positive maps} on $\A$. Namely, the elements of $\UCP(\A)$ are collections:
\begin{align}
\phi=\{\phi_I : I\in\cI\},
\end{align}
where each $\phi_I:\A(I)\rightarrow\A(I)$ is a normal faithful\footnote{Normality and faithfulness of $\phi_I$ follow from the vacuum preserving property (ii), \cf \cite[Proposition 3.1]{AcCe1982}, \cite[Section 3]{BaCaMu2016} and \cite[Theorem 1]{To1959}, because the vacuum state is normal and faithful on every local algebra by the Reeh--Schlieder theorem.} unital completely positive map such that:
\begin{enumerate}
\item[(i)] \textbf{Compatibility}: $\phi_{I_2} \restriction_{\A(I_1)}=\phi_{I_1}$ for every $I_1,I_2\in \cI$ with $I_1\subset I_2$.
\item[(ii)] \textbf{Vacuum preserving}: $\omega_I = \omega_I \circ \phi_I$ on $\A(I)$ for every $I\in\cI$, where $\omega_I := (\Omega, \slot \Omega)$ is the vacuum state restricted to $\A(I)$.
\item[(iii)] \textbf{Conformal symmetry}: $\Ad U(g) \circ \phi_I\circ \Ad U(g)^{-1}=\phi_{g I}$ for every $I\in\cI$, $g\in\mathrm{PSL}(2,\RR)$, and $\phi_I \restriction_{\Vir_\A(I)} = \id$
for every $I\in\cI$.
\end{enumerate}

We endow $\UCP(\A)$ with the coarsest topology such that the localization maps $\ell_I : \phi\mapsto \phi_I$ are continuous for every $I\in\cI$, where we consider the pointwise ultraweak operator topology, (also called \emph{BW topology}) on the set of unital completely positive maps $\A(I)\rightarrow\A(I)$ for fixed $I\in\cI$.
\end{defi}

Let $\phi\in\UCP(\A)$. Every $\phi_I$ is implemented by an operator $V_{\phi_I} : \Hil \to \Hil$, where $\Hil$ is the vacuum Hilbert space of $\A$, defined as the closure of the linear map:
\begin{align}\label{Vphi}
V_{\phi_I} x\Omega := \phi_I(x)\Omega, \quad x\in\A(I).
\end{align}
By the Reeh--Schlieder theorem, the vectors $x\Omega$ are dense in $\Hil$. By the Kadison--Schwarz inequality $\phi_I(x^*x) \geq \phi_I(x)^* \phi_I(x)$, \cite{Ka1952}, it follows that $V_{\phi_I}$ is bounded. Moreover, $V_{\phi_I} \Omega = \Omega$ and $\|V_{\phi_I}\| = 1$. See \cite[Section 2]{NiStZs2003}, \cite[Section 2.5]{BiDeGi2021}.

\begin{lem}\label{Vphiglobal}
Let $\phi\in\UCP(\A)$. The operators $V_{\phi_I}$ implementing $\phi_I$ are independent of $I\in\cI$. 

Hence we write $V_\phi := V_{\phi_I}$.
\end{lem}

\begin{proof}
By the Reeh--Schlieder theorem, the closed complex span of the vectors $x\Omega$, $x\in\A(I)$, is $\Hil$, irrespectively of the choice of $I\in\cI$.
If $I_1,I_2\in\cI$ with $I_1\subset I_2$, then $V_{\phi_{I_1}} = V_{\phi_{I_2}}$ by the compatibility requirement (i). If $I_1,I_2\in\cI$ have some overlap, take $I_3\in\cI$ with $I_3 \subset I_1, I_2$ and conclude as before. If $I_1,I_2\in\cI$ are arbitrary, increase one of the two until they overlap, hence $V_{\phi_{I_1}} = V_{\phi_{I_2}}$.
\end{proof}

In the notation of Section \ref{prelim}, the standard conditional expectation $E_{\Vir}$ of $\A$ onto $\Vir_\A$ belongs to $\UCP(\A)$. Its implementing operator \eqref{Vphi} is the Jones projection: $e_{\Vir} = V_{E_{\Vir}}$.

\begin{lem}\label{phiinUCPAisDiffcovar}
Let $\phi$ be a collection of maps as in Definition \ref{UCPA}, fulfilling all the properties except for the $\Vir_\A$-fixing condition (namely: $\phi_I \restriction_{\Vir_\A(I)} = \id$ for every $I\in\cI$), and let $I\in\cI$ be fixed. Then $\phi_I \restriction_{\Vir_\A(I)} = \id$ is equivalent to $\Vir_\A(I)$-bimodularity (namely: $\phi_I(xyz) = x\phi_I(y)z$ for every $x,z\in\Vir_\A(I)$, $y\in\A(I)$) and also equivalent to $V_{\phi_I}\in\Vir_\A(I)'$.

In particular, if $\phi \in \UCP(\A)$, then $U(\gamma) V_\phi U(\gamma)^{-1} = V_\phi$ for every $\gamma\in\Diff_+(S^1)$. 
\end{lem}

\begin{proof}
The non-trivial implication in the first equivalence follows from Choi's multiplicative domain theorem \cite{Ch1974}. For the second equivalence, assuming $\Vir_\A(I)$-bimodularity, observe that $V_{\phi_I} xy\Omega = \phi_I(xy)\Omega = x\phi_I(y)\Omega = xV_{\phi_I} y \Omega$ for every $x\in\Vir_\A(I)$, $y\in\A(I)$. Hence $V_{\phi_I} x = x V_{\phi_I}$ by the cyclicity of $\Omega$. Vice versa, $V_{\phi_I} x = x V_{\phi_I}$ implies that $\phi_I(xy) = x\phi_I(y)$ because $\Omega$ is separating.
The last statement follows from Lemma \ref{Vphiglobal} and from the fact that $\Diff_+(S^1)$ is generated by localized diffeomorphisms.
\end{proof}

Abstracting from the present setting, let $\M\subset\B(\Hil)$ be a von Neumann algebra and let $\Omega\in\Hil$ be a cyclic and separating unit vector for $\M$. Let $\omega := (\Omega, \slot \Omega)$ be the associated normal faithful state on $\M$. 
In \cite{BiDeGi2021}, \cite{BiDeGi2022}, we considered a notion of adjoint for $\omega$-preserving unital completely positive maps $\phi : \M \to \M$ with respect to $\omega$, introduced in \cite[Section 6]{AcCe1982}. When it exists, the \textbf{$\omega$-adjoint} of $\phi$ is the unique $\omega$-preserving unital completely positive map $\phi^\sharp:\M \to \M$ determined by the relation:
\begin{align}\label{equationACadjoint}
\omega(x\phi(y))=\omega(\phi^\sharp(x)y), \quad x,y\in\M.
\end{align}

The existence of the $\omega$-adjoint is characterized as follows in terms of the Tomita--Takesaki modular operator $\Delta$ and conjugation $J$ of $\M$ with respect to $\omega$:

\begin{prop}[{\cite[Proposition 6.1]{AcCe1982}}]\label{ACadjoint}
Let $\phi : \M \to \M$ be an $\omega$-preserving unital completely positive map.
Then the following are equivalent:
\begin{enumerate}
\item[$(1)$] $\phi$ admits an $\omega$-adjoint.
\item[$(2)$] $V_{\phi} \Delta^{it} = \Delta^{it} V_{\phi}$ for every $t\in\RR$.
\item[$(3)$] $V_{\phi} J = J V_{\phi}$.
\end{enumerate}
The condition (2) is equivalent to $\phi \circ \Ad \Delta^{it} = \Ad \Delta^{it} \circ \phi$ for every $t\in\RR$.
\end{prop}

Let now $\phi_I : \A(I) \to \A(I)$ be associated with $\phi\in\UCP(\A)$, for some $I\in\cI$, and let $\omega_I = (\Omega,\slot\Omega)$ be the vacuum state restricted to $\A(I)$ as before. By combining the Bisognano--Wichmann theorem with Proposition \ref{ACadjoint}, we have that $\phi_I$ is automatically $\omega_I$-adjointable:

\begin{lem}\label{dilationsiffadjointable}
Let $\phi\in\UCP(\A)$. Then the $\omega_I$-adjointability of each $\phi_I$ is guaranteed and equivalent to dilation covariance: 
$\Ad U(\delta_I(t)) \circ \phi_I\circ \Ad U(\delta_I(t))^{-1} = \phi_I$ for every $\delta_I(t) \in\PSL(2,\RR)$, $t\in\RR$.
\end{lem}

The operator $V_{\phi_I^\sharp} : \Hil \to \Hil$ implementing $\phi_I^\sharp$ as in \eqref{Vphi} is the Hilbert space adjoint $V_{\phi_I}^*$.\\ 
The collection of maps $\phi^\sharp := \{\phi_I^\sharp: I\in\cI\}$ defines an element of $\UCP(\A)$, as one can check using the definition of $\omega_I$-adjoint, or the properties of the implementing operator $V_{\phi^\sharp} = V_\phi^*$.

\begin{lem}\label{multiplicativeUCP}
Let $\phi\in\UCP(\A)$. If $\phi_I$ is multiplicative, \ie, $\phi_I(xy) = \phi_I(x)\phi_I(y)$ for every $x,y\in\A(I)$, then $\phi_I$ is an automorphism and $V_\phi$ is a unitary such that $V_\phi\Omega = \Omega$ and $\phi_I = \Ad {V_\phi}$ for every $I\in\cI$. Moreover, $\phi^\sharp = \phi^{-1}$.
\end{lem}

\begin{proof}
Recall from Lemma \ref{Vphiglobal} that $V_\phi = V_{\phi_I}$ is independent of $I\in\cI$.
If $\phi_I$ is multiplicative, then $V_{\phi_I}$ is an isometry, \ie, $V_{\phi_I}^*V_{\phi_I} = 1$, see, \eg, \cite[Lemma 2.1]{NiStZs2003}. Using that $\phi_I$ is $\omega_I$-preserving, $\omega_I$-adjointable, together with multiplicativity, we show that $\phi_I^\sharp \circ \phi_I = \phi_I \circ \phi_I^\sharp = \id_{\A(I)}$.  
By \eqref{equationACadjoint}
\begin{align}
\omega_I(x \phi_I^\sharp(\phi_I(y))) = \omega_I(\phi_I(x) \phi_I(y)) = \omega_I(\phi_I(xy)) = \omega_I(xy), \quad x,y\in\A(I),
\end{align}
hence $\phi_I^\sharp(\phi_I(y)) = y$. Similarly, $\phi_I(\phi_I^\sharp(y)) = y$. Thus $\phi_I$ is invertible with $\phi_I^\sharp = \phi_I^{-1}$ and $V_{\phi_I}$ is unitary. From the multiplicativity of $\phi_I$, it also follows that $\phi_I(x) = V_{\phi_I} x V_{\phi_I}^{-1}$ for every $x\in\A(I)$ as claimed.
\end{proof}

The special case inspected above, where $\phi\in\UCP(\A)$ and each $\phi_I$ is multiplicative, corresponds to the case of automorphisms of the conformal net considered in Definition \ref{AutA}. Obviously, the set of automorphisms has a group structure given by composition and inversion.

\begin{defi}\label{compoadjointUCPA}
There is a natural notion of \textbf{composition} on $\UCP(\A)$. If $\phi_1,\phi_2\in\UCP(\A)$, then
\begin{align}
\phi_1 \circ \phi_2 := \{{\phi_1}_I \circ {\phi_2}_I : I\in\cI\}
\end{align}
belongs to $\UCP(\A)$. The composition unit is given by $\id_\A := \{\id_{\A(I)} : I\in\cI\}$.
The previously defined \textbf{$\omega$-adjunction} $\phi^\sharp = \{\phi_I^\sharp: I\in\cI\}$ is involutive, \ie, $\phi^{\sharp\sharp} = \phi$, and $(\phi_1 \circ \phi_2)^\sharp = \phi_2^\sharp \circ \phi_1^\sharp$. This endows $\UCP(\A)$ with the structure of an \emph{monoid with involution}. 
Furthermore, $\UCP(\A)$ has a natural \emph{convex structure}, and the composition and $\omega$-adjunction operations are both \emph{affine maps}, \ie, they preserve convex combinations. 
\end{defi}

\begin{rmk}
The operations in $\UCP(\A)$ considered above have a natural interpretation in terms of quantum channels, in this case acting on each local algebra $\A(I)$. The composition corresponds to the usual concatenation of channels. The $\omega$-adjoint map with respect to the faithful vacuum state $\omega_I$ coincides in this special case with the Petz recovery map \cite{Pe1984}, \cite{Pet88}, \cf \cite[Chapter 12]{WildeBook}. Moreover, as the defining equation \eqref{equationACadjoint} suggests, the $\omega$-adjunction operation provides a quantum version of Bayes' theorem, studied in \cite{PaRu22}, \cite{GPRR23} in the case of finite-dimensional \Cstar-algebras and not necessarily faithful states.
\end{rmk}

We now come to the problem of extending a unital completely positive map on a single local algebra to the whole net:

\begin{lem}\label{localbijection}
Let $I\in\cI$ be fixed. Let $\phi_I : \A(I) \to \A(I)$ be a unital completely positive map\footnote{With abuse of notation, here we do not necessarily mean that the map $\phi_I$ is associated with $\phi\in\UCP(\A)$.} 
which is $\omega_I$-preserving, $\omega_I$-adjointable and $\Vir_\A(I)$-fixing. Then there is a collection of unital completely positive maps $\{\phi_J : J\in\cI\}$, $\phi_J:\A(J)\rightarrow\A(J)$, fulfilling (ii) and (iii) in Definition \ref{UCPA} and such that the map on $I$ is the initially prescribed map $\phi_I$.
\end{lem}

\begin{proof}
We have to show that $\phi_{g I} := \Ad U(g) \circ \phi_I \circ \Ad U(g)^{-1}$ for an arbitrary $g\in\PSL(2,\RR)$ chosen such that $gI = J$ for a fixed $J\in\cI$ is unambiguously defined and that the resulting collection of maps varying $J$ fulfills (ii) and (iii). First, $\phi_{gI}$ does not depend on the choice of $g$, because $\phi_I$ is $\omega_I$-adjointable and the set of transformations in $\PSL(2,\RR)$ that map $I$ onto itself coincides with the dilations of $I$, and applying Lemma \ref{dilationsiffadjointable}. Denote $\phi_{gI}$ by $\phi_J$. In particular, the map on $I$ is $\phi_I$. Each $\phi_J$ is clearly vacuum preserving and $\Vir_\A(J)$-fixing for every $J\in\cI$. We have to show that $\Ad U(h) \circ \phi_J\circ \Ad U(h)^{-1}=\phi_{hJ}$ for every $h\in\PSL(2,\RR)$, $J\in\cI$. Let $g\in\PSL(2,\RR)$ be such that $gI=J$. Then $\Ad U(h) \circ \phi_J\circ \Ad U(h)^{-1}=\Ad U(hg)  \circ \phi_I \circ \Ad U(hg)^{-1}=\phi_{hg I}=\phi_{h J}$.
\end{proof}

Recall from Definition \ref{UCPA} that the topology on $\UCP(\A)$ is the coarsest topology such that the localization maps $\ell_I : \phi\mapsto \phi_I$ are continuous for every $I\in\cI$, where the set of unital completely positive maps $\A(I)\rightarrow\A(I)$, for fixed $I$, is equipped with the pointwise ultraweak operator topology and denoted by $\UCP(\A(I))$.

\begin{thm}\label{UCPcompact}
$\UCP(\A)$ is a compact Hausdorff convex set.
\end{thm}

\begin{proof}
Let $\UCP(\A)^\wedge$ the set of collections of unital completely positive maps $\{\phi_I : I \in\cI\}$ fulfilling (ii) and (iii), but not necessarily the compatibility condition (i). Let $I\in\cI$ be fixed. By Lemma \ref{localbijection}, it follows that the localization map $\ell_I:\phi \in \UCP(\A)^\wedge\mapsto \phi_I$ is a bijection onto the set of $\omega_I$-preserving, $\omega_I$-adjointable and $\Vir_\A(I)$-fixing maps in $\UCP(\A(I))$. The latter set is compact and Hausdorff with the pointwise ultraweak operator topology, \cf \cite[Lemma 4.33]{BiDeGi2021}. 
Indeed, by \cite[Theorem 7.4]{Pa2002}, the set of all bounded linear operators from $\A(I)$ to $\B(\Hil)$ with norm at most 1 is compact by the Banach--Alaoglu theorem. Again by \cite[Theorem 7.4]{Pa2002}, the completely positive condition is closed in the pointwise ultraweak operator topology. Similarly for $\phi_I(1) = 1$ and $\phi_I(\A(I)) \subset \A(I)$. Also the conditions of $\omega_I$-preserving, $\omega_I$-adjointable and $\Vir_\A(I)$-fixing are closed. Explicitly, for the $\omega_I$-adjointability, let $({\phi_I})_\alpha$ be a net in $\UCP(\A(I))$ that converges to $\phi_I\in\UCP(\A(I))$. If every $({\phi_I})_\alpha$ is $\omega_I$-adjointable, or equivalently $\Ad U(\delta_I(t)) \circ ({\phi_I})_\alpha \circ \Ad U(\delta_I(t))^{-1} = ({\phi_I})_\alpha$ for every $t\in\RR$ by Lemma \ref{dilationsiffadjointable}, then in the limit $\Ad U(\delta_I(t)) \circ {\phi_I} \circ \Ad U(\delta_I(t))^{-1} = {\phi_I}$ for every $t\in\RR$, and $\phi_I$ is $\omega_I$-adjointable by the same lemma.

We now show that $\UCP(\A)^\wedge$ is also compact and Hausdorff when endowed with the same topology as for $\UCP(\A)$. Namely, the topology on $\UCP(\A)^\wedge$ is the smallest topology containing open sets of the type $\ell_J^{-1}(S)$, for some $J\in\cI$ and some open set $S$ of $\UCP(\A(J))$. The localization map $\ell_I$ is by definition continuous. We want to show that $\ell_I$ is a homeomorphism, \ie, an open map. Note that any open set of the type $\ell_J^{-1}(S)$, can be written as $\ell_I^{-1}(T)$ for an open set $T$ of $\UCP(\A(I))$. This follows as we can take $T := \ell_I(\ell_J^{-1}(S)) = \{(\ell_J^{-1}(\phi_J))_I: \phi_J\in S\}$. The set $T$ is open because, by the proof of Lemma \ref{localbijection}, $T = \{\Ad U(h) \circ \phi_J\circ \Ad U(h)^{-1}: \phi_J\in S\}$ for some $h\in\PSL(2,\RR)$ such that $hJ = I$.

To conclude the proof, we show that $\UCP(\A)$ is closed in $\UCP(\A)^\wedge$. Let $\phi_\alpha\in \UCP(\A)$ be a net with limit $\phi\in \UCP(\A)^\wedge$. By continuity of the localization maps, each $(\phi_\alpha)_I$ converges to $\phi_I$ for every $I\in\cI$. Let $I\subset J$ and $x\in\A(I)$, then $\phi_J(x)=\lim_\alpha(\phi_\alpha)_J(x)=\lim_\alpha(\phi_\alpha)_I(x)=\phi_I(x)$. Namely, the compatibility condition (i) holds and $\phi\in \UCP(\A)$.
\end{proof}

\begin{rmk}\label{AutUCP}
It is known that the set of automorphisms $\Aut(\A)$ (Definition \ref{AutA}) is compact with the strong operator topology on the implementing unitaries. This follows from the \emph{split property} which is a consequence of diffeomorphism covariance \cite[Theorem 5.4]{MoTaWe2018} together with \cite[Section 3]{DoLo1984}. The analogous statement for $\Aut(\A)$ with the induced pointwise ultraweak operator topology does not follow directly from the compactness of $\UCP(\A)$. For example, the pointwise ultraweak limit of automorphisms on a single von Neumann algebra need not be an automorphism. \Cf \cite[Section II]{DoHaRo1969I} for a proof of compactness of $\Aut(\A)$ in the context of 3+1 dimensional (graded) local and Poincar\'e covariant QFTs.
\end{rmk}

We shall be interested in unital completely positive maps on $\A$ that are \emph{extreme} in the sense of convex sets: $\phi = \lambda_1 \phi_1 + \lambda_2 \phi_2$ with $0< \lambda_1, \lambda_2 < 1$ and $\phi_1, \phi_2\in\UCP(\A)$, implies $\phi_1=\phi_2=\phi$.

\begin{defi}\label{QuOpA}
We define the \textbf{quantum operations} on the conformal net $(\A,U,\Omega)$ to be $\QuOp(\A) := \Extr(\UCP(\A))$, the set of \emph{extreme points} of $\UCP(\A)$.
\end{defi}

The following proposition says that quantum operations generalize automorphisms:

\begin{prop}\label{AutinQuOp}
$\Aut(\A)$ is contained in $\QuOp(\A)$. Furthermore, if $\phi\in\QuOp(\A)$, or more generally $\phi\in\UCP(\A)$, is invertible in $\UCP(\A)$, then  $\phi\in \Aut(\A)$ and $\phi^{\sharp}= \phi^{-1}$. 
In symbols, $\UCP(\A)^\times = \QuOp(\A)^\times = \Aut(\A)$.
\end{prop}

\begin{proof}
Automorphisms are extreme among unital completely positive maps by \cite[Theorem 1.4.6]{Ar1969}, \cf the proof of \cite[Corollary 4.50]{BiDeGi2021}. This proves the first statement. For the second statement, let $\phi^{-1}\in \UCP(\A)$ such that $\phi\circ\phi^{-1}=\phi^{-1}\circ\phi=\id_\A$ and let $I\in\cI$. For every unitary $u\in\A(I)$, by the Kadison--Schwarz inequality, we have:
\begin{align}
1 = \phi_I(\phi_I^{-1}(u^\ast u)) \geq \phi_I(\phi_I^{-1}(u^\ast)\phi_I^{-1}(u)) \geq \phi_I(\phi_I^{-1}(u^\ast))\phi_I(\phi_I^{-1}(u))=1
\end{align}
thus by $\phi_I^{-1}(\A(I))) = \A(I)$ and Choi's multiplicative domain theorem \cite{Ch1974}, $\phi_I$ is multiplicative. 
By Lemma \ref{multiplicativeUCP}, we have that $\phi\in\Aut(\A)$ and $\phi^\sharp = \phi^{-1}$ as desired.
\end{proof}

We shall later motivate our choice of extreme points in the set of unital completely positive maps as a definition of quantum operations (Definition \ref{QuOpA}), see at the end of Section \ref{SecRelQuOp}. The remainder of this section is dedicated to the study of fixed point subnets.

\begin{prop}\label{subnetsfromaction}
Let $S\subset\UCP(\A)$ be a subset, then the fixed point net $\A^S$ defined by setting $\A^S(I) := \{x\in\A(I) : \phi_I(x) = x \text{ for every } \phi\in S\}$, $I\in\cI$, is an irreducible conformal subnet of $\A$.
\end{prop}

\begin{proof}
As each $\phi_I$ is $\omega_I$-preserving and $\omega_I$ is faithful, $\A^{\{\phi\}}(I)\subset \A(I)$ is a von Neumann subalgebra for every $\phi\in S$, see, \eg, \cite[Theorem 2.3]{AGG02}. Thus 
\begin{align}
\A^S(I)&=\bigcap_{\phi\in S}\A^{\{\phi\}}(I)\subset \A(I)
\end{align}
is a von Neumann subalgebra. The irreducibility of the inclusion follows from $\Vir_\A(I)\subset \A^S(I)$ as in Proposition \ref{irrediffintermediate}. Isotony of the net $\A^S$ follows from the compatibility condition on the maps $\phi\in S$. M{\"o}bius covariance of the subnet follows from the M{\"o}bius covariance of the maps, while full conformal covariance follows from \cite[Proposition 3.7 (b)]{Ca2004}.
\end{proof}

The following is our first main result:

\begin{thm}\label{fixedpointsVir}
Let $(\A,U,\Omega)$ be a conformal net. Then $\A^{\QuOp(\A)} = \Vir_\A$.
\end{thm}

\begin{proof}
By Theorem \ref{UCPcompact}, $\UCP(\A)$ is a compact Hausdorff convex set, thus, by the Krein--Milman theorem, the standard conditional expectation $E_{\text{Vir}}$ of $\A$ onto $\Vir_\A$ belongs to the pointwise ultraweak closure of the convex span of $\QuOp(\A)$. Let $x\in \A^{\QuOp(\A)}(I)$, $I\in\cI$, \ie, $x\in\A(I)$ and $\phi_I(x)=x$ for every $\phi\in \QuOp(\A)$. Hence $x$ is fixed by arbitrary convex combinations of elements in $\QuOp(\A)$ and their pointwise ultraweak limits. This in turn implies that $(E_{\text{Vir}})_I(x)=x$, \ie, $x\in\Vir_\A(I)$.
\end{proof}

Theorem \ref{fixedpointsVir} above says that $\QuOp(\A)$ is an (in a sense minimal) extension of $\Aut(\A)$ that recovers the canonical minimal subnet $\Vir_\A$ from $\A$ as its fixed point subnet (or \emph{generalized orbifold} in the terminology of \cite{Bi2016}). Note that among the (finite index) irreducible extensions of $\Vir_c$, $c<1$, classified in \cite[Theorem 4.1]{KaLo2004}, there are examples of conformal nets (with index $3+\sqrt{3}$ over $\Vir_c$, \cf \cite[Theorem 2.3]{CaKaLo2010} for the possible small index values of arbitrary $\B\subset\A$) where $\Aut(\A)$ is the trivial group, hence $\A^{\Aut(\A)} = \A$. Nevertheless, $\Vir_\A \subset \A$ is non-trivial, \ie, $\A\neq \Vir_\A$.

\begin{rmk}\label{beyonddiscrete}
At this level of generality, we cannot say much neither about the induced topological structure nor about the algebraic structure of $\QuOp(\A)$, the set of all quantum operations on $\A$. Indeed, the extreme points of a compact convex set need not be closed, nor Borel, in general, see, \eg, \cite{Phe01}. Moreover, the composition of two quantum operations belongs of course to $\UCP(\A)$, but it need not be extreme. We shall come back to these two points in the special cases of finite index or irreducible discrete conformal inclusions in Section \ref{discrete}.
\end{rmk}

\section{Relative quantum operations}\label{SecRelQuOp}

In the previous section, we considered an arbitrary conformal net $(\A,U,\Omega)$ and the canonical conformal subnet $\Vir_\A\subset\A$ constructed from the diffeomorphism symmetries of $\A$. In this section, we consider more generally intermediate conformal nets $\Vir_\A \subset\B\subset\A$, or equivalently, by Proposition \ref{irrediffintermediate}, an arbitrary irreducible conformal inclusion $\B\subset\A$. 

\begin{defi}\label{relQuOp}
Let $\B\subset \A$ be an irreducible conformal inclusion. We define $\QuOp(\A|\B)$ to be the set of \textbf{quantum operations} on $\A$ \textbf{relative} to $\B$:
\begin{align}
\QuOp(\A|\B) := \left\{\phi\in\QuOp(\A) : \phi_I(x)=x \text{ for every } x\in\B(I),\, I\in\cI \right\} \! .
\end{align}
Namely, the set of extreme points of $\UCP(\A)$, as in Definition \ref{QuOpA}, that in addition fix $\B$ pointwise. 
Let also $\UCP(\A|\B) := \left\{\phi\in\UCP(\A) : \phi_I(x)=x \text{ for every } x\in\B(I),\, I\in\cI\right\}$.
\end{defi}

By definition, $\UCP(\A|\Vir_\A) = \UCP(\A)$ and $\QuOp(\A|\Vir_\A) = \QuOp(\A)$.

\begin{lem}\label{extremisnice}
$\QuOp(\A|\B)$ coincides with the set of extreme points of $\UCP(\A|\B)$.
\end{lem}

\begin{proof}
One inclusion is trivial, namely $\QuOp(\A|\B)$ is clearly contained and extreme in $\UCP(\A|\B)$.
Vice versa, let $\phi$ be extreme in $\UCP(\A|\B)$ and assume $\phi = \lambda_1 \phi_1 + \lambda_2 \phi_2$ with $0< \lambda_1, \lambda_2 < 1$, $\lambda_1 + \lambda_2 = 1$, and $\phi_1, \phi_2\in\UCP(\A)$. Denoted by $\iota_I$ the inclusion map of $\B(I)$ into $\A(I)$, we have $\iota_I = \phi_I\circ \iota_I =\lambda_1 (\phi_1)_I\circ \iota_I + \lambda_2 (\phi_2)_I\circ\iota_I$. By \cite[Theorem 1.4.6]{Ar1969}, \cf the proof of \cite[Lemma 4.49]{BiDeGi2021}, $\iota_I$ is extreme in the convex set of completely positive maps $\B(I)\rightarrow \A(I)$. Thus $\iota_I =(\phi_i)_I \circ \iota_I$, $i=1,2$, and therefore $\phi_1, \phi_2 \in \UCP(\A|\B)$. But since $\phi$ is extreme in $\UCP(\A|\B)$ by assumption, we get $\phi=\phi_1=\phi_2$. Thus $\phi \in \QuOp(\A|\B)$.
\end{proof}

The relative versions of Proposition \ref{AutinQuOp} and Theorem \ref{fixedpointsVir} hold:

\begin{prop}
Let $\Aut(\A|\B) := \left\{\alpha\in\Aut(\A) : \alpha(x)=x \text{ for every } x\in\B(I),\, I\in\cI\right\}$. Then $\Aut(\A|\B)\subset\QuOp(\A|\B)$. 
Furthermore, $\UCP(\A|\B)^\times = \QuOp(\A|\B)^\times = \Aut(\A|\B)$.
\end{prop}

\begin{proof}
Immediate from the proof of Proposition \ref{AutinQuOp}.
\end{proof}

\begin{thm}\label{fixedpointsB}
Let $\B\subset\A$ be an irreducible conformal inclusion. Then $\B=\A^{\QuOp(\A|\B)}$.
\end{thm}

\begin{proof}
The set $\UCP(\A|\B)$ is convex and pointwise ultraweakly closed in $\UCP(\A)$, hence compact and Hausdorff with the induced topology by Theorem \ref{UCPcompact}. The rest follows as in Theorem \ref{fixedpointsVir}.
\end{proof}

As for the examples with $\B=\Vir_\A$ from \cite{KaLo2004} mentioned in the previous section, there are examples of non-trivial irreducible conformal inclusions $\B\subset\A$ with $\Aut(\A|\B)$ the trivial group, but $\B\neq\A$. See, \eg, \cite[Example 2]{Bi2019MFO}, where $\B = L\SU(2)_{10} \subset \A = L\Spin(5)_1$ with index $3+\sqrt{3}$. Further somehow opposite or intermediate examples (with finite or infinite index) are as follows. If $\A := L\SU(2)_1$ then $\Aut(\A) = \SO(3)$ and ${\A}^{\Aut(\A)} = \Vir_{\A}$ with central charge $c=1$. Every intermediate conformal net $\Vir_{\A} \subset \B \subset \A$ corresponds to a closed subgroup $H\subset\SO(3)$ via $\B = {\A}^H$, see \cite[Section 3]{Ca2004}, \cite[Section 4]{Xu2005}. If $\A_N := L\U(1)_{2N}$, for certain integer values of $N$, namely $N \geq 2$ and not a perfect square, the intermediate conformal nets $\Vir_{\A_N} \subset\B\subset\A_N$ are either $\Vir_{\A_N}$, again with central charge $c=1$, or $\B = \A_N^H$ for a closed subgroup $H\subset\Aut(\A_N) = D_\infty$, where $D_\infty \cong \TT \rtimes \ZZ_2$ is the infinite dihedral group, and $\Vir_{\A_N} \subsetneq \A_N^{D_\infty}$. The case $N=2$ also provides examples of conformal subnets $\B\subset \A_2$ that are not irreducible, \ie, by Proposition \ref{irrediffintermediate}, not intermediate $\Vir_{\A_2} \subset\B\subset\A_2$. In fact, $\B$ is isomorphic to $\Vir_{1/2}$ in all these latter examples. See \cite[Corollary 3.9, Theorem 4.4]{CaGaHi2019}.
\smallskip

We conclude this section with a comment on (finite or compact) group fixed points. This is our motivation for considering the extreme points of $\UCP(\A)$ and $\UCP(\A|\B)$ as a definition of quantum operations (Definition \ref{QuOpA} and \ref{relQuOp}). The proof is a combination of \cite[Proposition 9.2]{BiDeGi2021} and Theorem \ref{discretecase} in the next section. 

\begin{cor}
Let $G$ be a finite or compact metrizable group acting properly on a conformal net $(\A,U,\Omega)$ as in Definition \ref{groupaction}. Then $\QuOp(\A|\A^G) = \Aut(\A|\A^G)\cong G$, where the isomorphism (as topological groups) is given by the action of $G$ on $\A$. Moreover, $\UCP(\A|\A^G) \cong P(G)$, where $P(G)$ is the convex set of (positive) probability Radon measures on $G$.
\end{cor}

\section{$\QuOp(\A|\B)$ for discrete subnets}\label{discrete}

Let $(\A,U,\Omega)$ be a conformal net and $\B\subset \A$ an irreducible and \emph{discrete} (Definition \ref{irredinclusion} and \ref{discretesubnets}) or \emph{finite index} (Definition \ref{finindexsubnets}) conformal subnet. Under these assumptions, we already recalled at the end of Section \ref{RepnConfNets} that $\B(I)\subset\A(I)$ is an irreducible discrete and (braided with respect to the DHR braiding) local type $\III$ subfactor, for every $I\in\cI$. We begin by showing (Theorem \ref{discretecase}) that every $\B(I)$-fixing unital completely positive map $\A(I)\to\A(I)$, with no additional assumption, can be extended to the whole net and defines an element of $\UCP(\A|\B)$. 

\begin{defi}
Let $I\in\cI$ be fixed. Let $\UCP(\A(I)|\B(I))$ be the set of $\B(I)$-fixing unital completely positive maps $\phi:\A(I)\to\A(I)$,
endowed with the pointwise ultraweak operator topology. 
\end{defi}

\begin{lem}[{\cite[Corollary 4.25]{BiDeGi2021}}]\label{adjauto}
Every $\phi_I \in \UCP(\A(I)|\B(I))$\footnote{As already done in Section \ref{QuOponConfNets}, Lemma \ref{localbijection}, with abuse of notation we denote by $\phi_I$ the elements of $\UCP(\A(I)|\B(I))$, not necessarily coming from elements of $\UCP(\A)$.} is automatically $\B(I)$-bimodular $\omega_I$-preserving 
and $\omega_I$-adjointable, where $\omega_I = (\Omega,\slot\Omega)$ is the vacuum state\footnote{Or any normal faithful state on $\A(I)$ which is invariant for the vacuum preserving conditional expectation $E_I:\A(I)\to \B(I)\subset\A(I)$. Recall that $E_I$ is unique among normal faithful conditional expectations by Remark \ref{irreduniqueE}.} on $\A(I)$.
\end{lem}

The following lemma should be compared with Lemma \ref{phiinUCPAisDiffcovar}, where $\phi$ belongs to $\UCP(\A)$.

\begin{lem}\label{compatibility}
Let $\phi_I \in \UCP(\A(I)|\B(I))$ and let $V_{\phi_I}$ be as in \eqref{Vphi}. Then $U(\gamma) V_{\phi_I} U(\gamma)^{-1} = V_{\phi_I}$ for every $\gamma\in\Diff_+(S^1)$.
\end{lem}

\begin{proof}
By Lemma \ref{adjauto}, $\phi_I$ is automatically $\B(I)$-bimodular and $\omega_I$-adjointable. We argue that $V_{\phi_I}\in \A_1(I)\cap \B(I)^\prime$, where $\A_1(I)$ is the von Neumann algebra generated by $\A(I)$ and by the Jones projection $e_I$ of $\A$ onto $\B$ as in \eqref{Jonesproj}. Namely, $\A_1(I)$ is the Jones extension of $\A(I)$ with respect to $\B(I)$. By $\B(I)$-bimodularity, clearly $V_{\phi_I}\in \B(I)^\prime$. By Proposition \ref{ACadjoint}, $V_{\phi_I}$ commutes with the Tomita conjugation $J$ of $\A(I)$ with respect to the vacuum vector, hence $V_{\phi_I}\in J\B(I)^\prime J=\A_1(I)$. Consequently, $V_{\phi_I}\in \A_1(I)\cap \B(I)^\prime\subset \B(I^\prime)^\prime\cap \B(I)^\prime=(\B(I^\prime)\vee \B(I))^\prime$, where the inclusion follows from relative locality, \ie, $\A(I)\subset \B(I^\prime)^\prime$, and $e_I=e_{I^\prime}\in \B(I^\prime)^\prime$. Note that the commutants and the von Neumann algebra generated are taken in the vacuum Hilbert space of $\A$.
By \cite[Proposition 3.3 (a)]{Ca2004}, using the discreteness assumption, we thus also have $V_{\phi_I}\in (\bigvee_{I\in\cI}\B(I))'$. By irreducibility and by Proposition \ref{irrediffintermediate}, we have $\Vir_\A\subset\B$. Thus $V_{\phi_I}\in (\bigvee_{I\in\cI}\Vir_\A(I))^{\prime}$ and $U(\gamma)V_{\phi_I}U(\gamma)^{-1}=V_{\phi_I}$ for every $\gamma\in\Diff_+(S^1)$, as $\Diff_+(S^1)$ is generated by localized diffeomorphisms.
\end{proof}

As a consequence of these two lemmas, for irreducible discrete conformal inclusions, the structure of $\UCP(\A|\B)$ (Definition \ref{relQuOp}) is completely determined by a single subfactor $\B(I)\subset\A(I)$:

\begin{thm}\label{discretecase}
Let $\B\subset\A$ be an irreducible discrete conformal inclusion. Let $I\in\cI$ be fixed. Then the map: 
\begin{align}
\UCP(\A|\B)&\rightarrow  \UCP(\A(I)|\B(I))\\
\phi&\mapsto\phi_I
\end{align}
is an affine homeomorphism of convex topological spaces.

In particular, the extreme points are homeomorphic: $\QuOp(\A|\B) \cong \Extr(\UCP(\A(I)|\B(I)))$.
\end{thm}

\begin{proof}
The map $\phi\mapsto \phi_I$ is injective by M{\"o}bius covariance of the map $\phi$ as in Lemma \ref{localbijection}. To show surjectivity, given $\phi_I\in\UCP(\A(I)|\B(I)$, by Lemma \ref{adjauto}, we can use the same proof of  Lemma \ref{localbijection} to contruct a collection of unital completely positive maps $\{\phi_J : J\in\cI\}$ fulfilling (ii) and (iii) in Definition \ref{UCPA} and such that $\phi_I$ is the prescribed map.
Namely, $\phi_{J} := \Ad U(g) \circ \phi_I \circ \Ad U(g)^{-1}$ for $g\in\PSL(2,\RR)$ such that $gI = J$, irrespectively of the choice of $g$. 
By Lemma \ref{compatibility}, the collection of maps $\{\phi_J : J\in\cI\}$ is compatible, \ie, it fulfills (i) in Definition \ref{UCPA}.  
Indeed, if $x\in\A(J)$, then $V_{\phi_J} x \Omega = \phi_J(x) \Omega = U(g) \phi_I(U(g)^{-1} x U(g)) U(g)^{-1} \Omega = U(g) V_{\phi_I} U(g)^{-1} x \Omega$ because $U(h) \Omega = \Omega$ for every $h\in\PSL(2,\RR)$ and $U(g)^{-1} x U(g)\in\A(I)$. By the cyclicity of $\Omega$, $V_{\phi_J} = U(g) V_{\phi_I} U(g)^{-1}$, hence $V_{\phi_J} = V_{\phi_I}$ by Lemma \ref{compatibility} for every $J\in\cI$. This immediately entails the compatibility condition when $J_1,J_2\in\cI$ are such that $J_1\subset J_2$. Thus $\{\phi_J : J\in\cI\}\in\UCP(\A|\B)$. The fact that the map $\phi\mapsto \phi_I$ is a homeomorphism follows as in the proof of Theorem \ref{UCPcompact}.
\end{proof}

In \cite{BiDeGi2021}, we showed that for an irreducible discrete local type $\III$ subfactor $\N\subset\M$, the set $\Extr(\UCP(\M|\N))$, therein denoted by $K(\N\subset\M)$, is closed (hence compact) in the pointwise ultraweak operator topology and it has the structure of a \emph{compact hypergroup} \cite[Definition 3.2]{BiDeGi2021}. Our notion is inspired and closely related to other notions of hypergroup/hypercomplex system present in the literature \cite{BlHe1995}, \cite{BeKa1998}. We recall below our definition. For finite sets it reduces to the purely algebraic notion of finite hypergroup, see, \eg, \cite{SuWi2003}. Moreover, every finite or compact group is also a finite or compact hypergroup, \cf \cite[Example 3.5]{BiDeGi2021}.

\begin{defi}\label{abstractcpthyp}
Let $K$ be a compact Hausdorff space. Denote by $P(K)$ the convex set of (positive) probability Radon measures on $K$, by $C(K)$ the algebra of continuous complex valued functions on $K$, and by $\delta_x$ the normalized Dirac measure on $x$. Then $K$ is called a \textbf{compact hypergroup} if it is equipped with a biaffine operation called \emph{convolution}:
\begin{align}
P(K)\times P(K)\to P(K), \quad (\mu,\nu) \mapsto \mu\ast\nu,
\end{align}
with an \emph{involution} $K \to K, x \mapsto x^\sharp$, and with an \emph{identity element} $e\in K$, fulfilling:
\begin{enumerate}
\item[(i)] $P(K)$ is a monoid with involution with respect to $\ast, \,^\sharp, \delta_e$, where the involution on $P(K)$ is defined by $\mu^\sharp(E) := \mu(E^\sharp)$ for every Borel set $E\subset K$.
\item[(ii)] The involution $x \mapsto x^\sharp$ is continuous on $K$ and the map:
\begin{align}
(x,y)\in K\times K \mapsto \delta_x\ast\delta_y \in P(K)
\end{align}
is jointly continuous with respect to the weak* topology on measures.
\item[(iii)] There exists a (unique) faithful probability measure $\mu_K\in P(K)$, called a \emph{Haar measure},
such that for every $f,g\in C(K)$ and $y\in K$ it holds:
\begin{align}
\int_Kf(y\ast x) g(x)\dd\mu_K(x) &= \int_Kf(x)g(y^\sharp \ast x) \dd\mu_K(x), \\
\int_Kf(x\ast y) g(x)\dd\mu_K(x) &= \int_Kf(x)g(x \ast y^\sharp) \dd\mu_K(x),
\end{align}
where
\begin{align}
f(x\ast y ) := (\delta_x\ast\delta_y)(f) = \int_Kf(z)\dd(\delta_x \ast \delta_y)(z).
\end{align}
\end{enumerate}
\end{defi}

For later use, we recall \cite[Definition 5.1]{BiDeGi2022}, \cf \cite[Definition 1.5.1]{BlHe1995}:

\begin{defi}
A \textbf{closed subhypergroup} of $K$ is a closed subset $H\subset K$ which is closed under the operations of $K$: $\delta_x\ast\delta_y$ is supported on $H$ for every $x,y\in H$, and $x^\sharp \in H$, $e\in H$, and which admits a Haar measure in $P(H)$. In particular, $H$ is a compact hypergroup as well.
\end{defi}

We also need the following notion from convex analysis, see, \eg, \cite{AlfBook}, \cite{AlfShuBook}: 

\begin{defi}
A compact subset $X$ of a locally convex space is said to be a \textbf{Bauer simplex} if the extreme points $\Extr(X)$ are closed (hence compact) in $X$, and, for every point $x\in X$, there exists a unique (positive) probability Radon measure $\mu_x$ supported on Extr$(X)$ whose barycenter is $x$, \ie, for every affine function $f:X\rightarrow \RR$, it holds $\int_X f d\mu_x = f(x)$.
\end{defi}

By combining Theorem \ref{discretecase} above with \cite{BiDeGi2021},  we can argue that for an irreducible discrete inclusion of conformal nets $\B\subset\A$ the set of quantum operations $\QuOp(\A|\B)$ has naturally the structure of a compact hypergroup, and that $\UCP(\A|\B)$ is a Bauer simplex:

\begin{thm}\label{QuOpHyper}
Let $\B\subset\A$ be an irreducible discrete conformal inclusion. Then
\begin{enumerate}
\item[$(1)$] The set $\UCP(\A|\B)$ is a Bauer simplex. In particular, $\UCP(\A|\B)$ is affinely homeomorphic to $P(\QuOp(\A|\B))$ via the map $\phi\mapsto \mu_{\phi}$. 
\smallskip
\item[$(2)$] The subset $\QuOp(\A|\B)$ is a compact hypergroup with the following operations:
\begin{itemize}
\item The convolution is induced by the composition in $\UCP(\A|\B)$.
\item The involution is induced by the $\omega$-adjunction.
\item The Haar measure is the unique probability Radon measure supported on $\QuOp(\A|\B)$ with barycenter $E_\B\in \UCP(\A|\B)$, the standard conditional expectation of $\A$ onto $\B$. 
\end{itemize}
\end{enumerate}
\end{thm}

\begin{proof}
The first statement is the combination of Theorem \ref{discretecase} with \cite[Theorem 4.34]{BiDeGi2021}. The second part of the first statement is a characterization of Bauer simplex, see, \eg, \cite[Proposition 1.1]{Phe01}. The second statement follows from Theorem \ref{discretecase} and \cite[Theorem 4.51]{BiDeGi2021}.
\end{proof}

As for compact groups of automorphisms, for the canonical compact hypergroup of quantum operations the defining representation $\iota$ of $\B$ on the vacuum Hilbert space of $\A$ can be identified with the regular representation. We recall, \eg, from \cite[Chapter 2]{BlHe1995}, \cite[Section 6]{BiDeGi2021}:

\begin{defi}\label{cpthyprep}
Let $K$ be a compact hypergroup and let $M(K)$ be the associated unital involutive Banach algebra of complex Radon measures. A \textbf{representation} $\pi$ of $K$ on a Hilbert space $\Hil_\pi$ is a unital involutive algebra morphism $\pi: M(K) \to \B(\Hil_\pi)$. Note that $\pi$ is automatically norm decreasing, $\|\pi(\mu)\| \leq \|\mu\|$. A representation $\pi$ is called \emph{continuous} if its restriction to the positive measures is continuous from the weak* topology to the weak operator topology.
\end{defi}

The following two results are a combination of Theorem \ref{QuOpHyper} with \cite[Theorem 6.4]{BiDeGi2021} and \cite[Theorem 6.5]{BiDeGi2021}, respectively, to which we refer for further details.

\begin{thm}
Let $\B\subset\A$ be an irreducible discrete conformal inclusion. Then $\iota$ can be identified with the direct sum of all continuous irreducible representations $\pi$ of the compact hypergroup $\QuOp(\A|\B)$, each counted with multiplicity equal to $\dim(\Hil_\pi)$. 
\end{thm}

Besides the usual vector space notion of dimension of $\Hil_\pi$, a representation of a compact hypergroup has its own notion of dimension, called \textbf{hyperdimension}, introduced by Vrem \cite{Vr1979} and denoted by $k_\pi$, \cf \cite[Section 6]{BiDeGi2021}. In general, $\dim(\Hil_\pi) \leq k_\pi$.

\begin{thm}
The hyperdimension of each $\pi$ equals the tensor \Cstar-categorical dimension of the associated irreducible representation $\rho_\pi$ appearing in the decomposition of $\iota$. In symbols, $k_\pi = d(\rho_\pi)$.  
\end{thm}

\section{Galois correspondence}

In this section, we establish a Galois-type correspondence between irreducible (Definition \ref{irredinclusion}) conformal subnets of a given conformal net $\A$ (Definition \ref{confnet} and \ref{confsubnetdef}) and suitable subsets of all quantum operations on $\A$ (Definition \ref{QuOpA}), mainly in the discrete case (Definition \ref{discretesubnets}). 

Let $\B\subset\A$ be an irreducible conformal inclusion. $\UCP(\A)$ and $\UCP(\A|\B)$ (Definition \ref{UCPA} and \ref{relQuOp}) are non-empty convex compact Hausdorff spaces (Theorem \ref{UCPcompact}) and also monoids with involution with respect to the (affine) composition and $\omega$-adjunction operations (Definition \ref{compoadjointUCPA}). Recall from Proposition \ref{irrediffintermediate} that $\B\subset\A$ irreducible is equivalent to $\Vir_\A \subset\B\subset\A$, where $\Vir_\A$ is the Virasoro subnet of $\A$.

\begin{lem}\label{stdEisHaarelement}
 The standard conditional expectation $E_{\Vir}$ of $\A$ onto $\Vir_\A$ is a self-involutive projection in $\UCP(\A)$: $E_{\Vir} = E_{\Vir}^\sharp = E_{\Vir} \circ E_{\Vir}$, with the absorption property: $\phi \circ E_{\Vir} = E_{\Vir} = E_{\Vir} \circ \phi$ for every other $\phi\in\UCP(\A)$. Similarly for
the standard conditional expectation $E_\B$ of $\A$ onto $\B$ in $\UCP(\A|\B)$. 
\end{lem}

\begin{proof}
We only need to show the absorption property. $\phi \circ E_{\Vir} = E_{\Vir}$ is immediate from the fact that $\phi$ fixes $\Vir_A$ by definition and $E_{\Vir}$ projects onto it. Using this and $E_{\Vir} = E_{\Vir} \circ E_{\Vir}$, we have $E_{\Vir} \circ \phi \circ E_{\Vir} \circ \phi = E_{\Vir} \circ \phi$, thus $E_{\Vir} \circ \phi = E_{\Vir}$ by irreducibility and uniqueness of the standard conditional expectation, or because $E_{\Vir} \circ \phi$ preserves the vacuum state as well.
\end{proof}

\begin{defi}
We call \textbf{Haar element} a self-involutive projection with the absorption property in a monoid with involution. When a Haar element exists, it is obviously necessarily unique.
\end{defi}

\begin{prop}\label{genGalois}
Let $\A$ be a conformal net. There is a bijective correspondence between:
\begin{enumerate}
\item[$(1)$] Irreducible conformal subnets $\C\subset\A$.
\item[$(2)$] Self-involutive projections in $\UCP(\A)$.
\item[$(3)$] Closed convex submonoids with involution of $\UCP(\A)$ with Haar element, 
that are maximal among all closed convex submonoids with involution with the same Haar element.
\end{enumerate}

The correspondence from $(1)$ to $(2)$ to $(3)$ is given by $\C \mapsto E_{\C} \mapsto \{\phi\in\UCP(\A): \phi\circ E_{\C} = E_\C = E_{\C}\circ \phi \}$ with Haar element $E_\C$, the standard conditional expectation of $\A$ onto $\C$. From $(1)$ to $(3)$, it also holds $\{\phi\in\UCP(\A): \phi\circ E_{\C} = E_\C = E_{\C}\circ \phi\} = \UCP(\A|\C)$.
\smallskip

More generally, let $\B\subset\A$ be an irreducible conformal inclusion. There is a bijective correspondence between intermediate conformal nets $\B\subset \C\subset\A$ in $(1)$, self-involutive projections in $\UCP(\A|\B)$ in $(2)$, and closed convex submonoids with involution of $\UCP(\A|\B)$ with Haar element, that are maximal among all closed convex submonoids with involution with the same Haar element in $(3)$.
\end{prop}

\begin{proof}
We only prove the first statement, where $\B=\Vir_\A  \subset\A$. The second more general statement follows analogously.
First, the map $\C\mapsto E_{\C}$ is clearly injective since if $E_{\C_1}=E_{\C_2}$ then $\C_1=\C_2$ as the subnets are the fixed points of the respective conditional expectations. It is also surjective, since given $E=E^{\sharp}=E\circ E\in\UCP(\A)$, then $\A^{\{E\}}$ is an irreducible conformal subnet by Proposition \ref{subnetsfromaction}. Moreover, $E_I$ is a conditional expectation of $\A(I)$ onto its fixed point subalgebra $\A^{\{E\}}(I)$ by definition, for each $I\in\cI$, hence it is unique by irreducibility and it must coincide with $(E_{\A^{\{E\}}})_I$. 
Second, given $E=E^{\sharp}=E\circ E \in \UCP(\A)$, the map $E\mapsto \{\phi\in\UCP(\A): \phi\circ E=E=E\circ \phi\}$ is injective since if $E_1, E_2$ as above give rise to the same subset, to which they both belong, then in particular $E_1=E_2\circ E_1=E_2$. Moreover, $E$ is by definition a Haar element for $\{\phi\in\UCP(\A): \phi\circ E=E=E\circ \phi\}$. This second map is obviously surjective onto closed convex submonoids with involution of $\UCP(\A)$ with Haar element, which are maximal among closed convex submonoids with involution with the same Haar element.
Lastly, the inclusion $\{\phi\in\UCP(\A): \phi\circ E_{\C} = E_\C = E_{\C}\circ \phi\} \subset \UCP(\A|\C)$ is immediate, while the opposite inclusion follows from Lemma \ref{stdEisHaarelement} with $\C$ in place of $\B$.
\end{proof}

The following is our second main result. In particular, the second statement in Theorem \ref{Galois} below generalizes a result of Longo, \cite[Theorem 21]{Lo2003}, from finite index (Definition \ref{finindexsubnets}) to irreducible discrete conformal subnets (Definition \ref{discretesubnets}):

\begin{thm}\label{Galois}
Let $\A$ be a conformal net and let $\B\subset\A$ be a conformal subnet such that the inclusion is irreducible and discrete. There is a bijective correspondence between:
\begin{align}
\left\{ \C \subset \A \,|\, \text{conformal subnet with }\B\subset\C \right\} \longleftrightarrow \left\{ K\subset \QuOp(\A|\B)\,|\, K \text{ closed subhypergroup}\right\}
\end{align}
The correspondence is given by $\C\mapsto K_{\C} := \QuOp(\A | \C)$ and $K\mapsto \A^K$.
\smallskip

Furthermore, for fixed $I\in\cI$, there is a bijective correspondence between intermediate von Neumann algebras $\B(I)\subset \N\subset \A(I)$ and intermediate conformal nets $\B \subset\C \subset \A$ such that $\C(I) = \N$.
\end{thm}

\begin{proof}
We first show the second statement. Let $\B(I)\subset \N\subset \A(I)$ be an intermediate von Neumann algebra. By \cite[Theorem 4.5]{BiDeGi2022}, $\N\subset \A(I)$ is an irreducible discrete local type $\III$ subfactor, and thus by \cite[Theorem 5.2]{BiDeGi2022}, $\Extr(\UCP(\A(I)|\N))$ (therein denoted by $K(\N\subset \A(I))$) 
is a closed subhypergroup of $\Extr(\UCP(\A(I)|\B(I)))$. The latter is homeomorphic to $\QuOp(\A|\B)$ by Theorem \ref{discretecase}, where the hard part is to extend maps on a single local algebra $\A(I)$ to a compatible family of maps on the whole net $\A$. Moreover, as the hypergroup operations are defined on each local algebra by Theorem \ref{QuOpHyper} and \cite[Theorem 4.51]{BiDeGi2021}, the homeomorphism between $\Extr(\UCP(\A(I)|\B(I)))$ and $\QuOp(\A|\B)$ is an isomorphism of compact hypergroups which intertwines the respective actions on $\A(I)$ and $\A$. In particular, there is a copy, denoted by $K$, of $\Extr(\UCP(\A(I)|\N))$ inside $\QuOp(\A|\B)$. By Proposition \ref{subnetsfromaction} and Theorem \ref{fixedpointsB}, $\C := \A^K$ is an intermediate conformal net and $\C(I) = \N$ by \cite[Theorem 5.7]{BiDeGi2021} as desired. With this definition of $\C$, we have that $K = \QuOp(\A|\C)$ is a closed subhypergroup of $\QuOp(\A | \B)$. 

Since the map $\C\mapsto \C(I)$ between intermediate conformal nets and intermediate von Neumann algebras is bijective by conformal covariance, which is fixed by $\A$, and since the map $\C(I)\mapsto \Extr(\UCP(\A(I)|\C(I)))$ between intermediate von Neumann algebras and closed subhypergroups of $\Extr(\UCP(\A(I)|\B(I)))$ is bijective by \cite[Theorem 5.2]{BiDeGi2022}, we have that the map $\C\mapsto K_{\C} := \QuOp(\A|\C)$, contained in $\QuOp(\A|\B)$, is also bijective. Thus the proof is complete.
\end{proof}

\begin{rmk}
An alternative proof of the second statement in Theorem \ref{Galois}, more in line with the proof of \cite[Theorem 21]{Lo2003}, goes as follows. The idea is to view $\N$ as being generated by $\B(I)$ and by a subset of a Pimsner--Popa basis of global charged fields for $\A$ over $\B$, and to establish the M\"obius covariance of the latter using Connes cocycles. This alternative proof, which we don't report here, makes use of results in \cite[Section 3]{IzLoPo1998}, \cite[Section 1 and 2]{Lon97}, \cite[Section 6]{DeGi2018}. Here we prefer to stick to the idea of viewing $\N$ as fixed point subalgebra of $\A(I)$ under a subset of quantum operations in $\QuOp(\A|\B)$, and to use the results of Section \ref{discrete}. Both arguments use the discreteness assumption on $\B\subset\A$.
\end{rmk}

We now turn to the finite index case (Definition \ref{finindexsubnets}). Note that if a conformal net $\A$ has central charge $c<1$, these nets are classified in \cite{KaLo2004}, then $\Vir_\A$ is \emph{completely rational} in the sense of \cite{KaLoMg2001}. Hence $\Vir_\A \subset \A$ has finite index, $\A$ is completely rational, and $\QuOp(\A)$ is a finite hypergroup, see Theorem \ref{Galoisfinindex} below. However, this is no longer the case even if $\A$ is completely rational with central charge $c\geq 1$, as the examples $\A = L\SU(2)_1$ or $\A_N = L\U(1)_{2N}$, with $N \geq 2$, mentioned at the end of Section \ref{SecRelQuOp} already show. Both families of examples have central charge $c=1$. In the first, $\QuOp(\A) = \Aut(\A) \cong SO(3)$. In the second, if $N=k^2$ for some integer $k$, then $\QuOp(\A_N) \cong \SO(3)\CS\ZZ_k$ (the double coset compact hypergroup associated with $\ZZ_k \subset \SO(3)$, see \cite[Theorem 1.1.9]{BlHe1995} and \cite[Section 9.2]{BiDeGi2021}). If $N$ is not a perfect square, then $\QuOp(\A_N) \supsetneq \Aut(\A_N) \cong D_\infty$ (the infinite dihedral group).
In all these cases, $\A$ is completely rational and $\QuOp(\A)$ is an infinite set. For a conformal net $\A$, the inclusion $\Vir_\A\subset\A$ neither has finite index nor it is discrete in general \cite{Fr1993}, \cite{Re1994}, \cite{Ca2003}, and we don't expect $\QuOp(\A)$ to be compact (nor locally compact) with the natural topology induced from $\UCP(\A)$, unless $\Vir_\A\subset\A$ is discrete. \Cf Remark \ref{beyonddiscrete}.
\medskip

The following is a restatement of the first main result in \cite{Bi2016}:

\begin{thm}[{\cite[Theorem 1.3, see also Theorem 3.8 and 4.22]{Bi2016}}]\label{Galoisfinindex}
Let $\A$ be a conformal net.
There is a bijective correspondence between: 
\begin{align}  
\left\{\B\subset \A \,|\, \text{conformal subnet with } [\A:\B]<\infty\right\} \longleftrightarrow \left\{ K\subset \QuOp(\A)\,|\, K \text{ finite hypergroup}\right\}
\end{align}
The correspondence is given by $\B\mapsto K_\B := \QuOp(\A | \B)$ and $K\mapsto \A^K$.
\end{thm}

\bigskip
\noindent
{\bf Acknowledgements.}
We are indebted to Sebastiano Carpi for several comments on a previous version of the manuscript, which we incorporated in the proof of Proposition \ref{irrediffintermediate} and in Remark \ref{AutUCP}, and to Tiziano Gaudio for drawing our attention to an interesting class of conformal inclusions. We also thank them for stimulating discussions on the topics treated in this paper.

\bigskip

\def\cprime{$'$}\newcommand{\noopsort}[1]{}

\bigskip

\address


\begin{bibdiv}
\begin{biblist}

\bib{AcCe1982}{article}{
      author={Accardi, Luigi},
      author={Cecchini, Carlo},
       title={Conditional expectations in von {N}eumann algebras and a theorem
  	of {T}akesaki},
        date={1982},
        ISSN={0022-1236},
     journal={J. Funct. Anal.},
      volume={45},
       pages={245\ndash 273},
         url={http://dx.doi.org/10.1016/0022-1236(82)90022-2},
}

\bib{AGG02}{article}{
      author={Arias, A.},
      author={Gheondea, A.},
      author={Gudder, S.},
       title={Fixed points of quantum operations},
        date={2002},
        ISSN={0022-2488},
     journal={J. Math. Phys.},
      volume={43},
       pages={5872\ndash 5881},
         url={https://doi.org/10.1063/1.1519669},
}

\bib{AlfBook}{book}{
      author={Alfsen, Erik~M.},
       title={Compact convex sets and boundary integrals},
      series={Ergebnisse der Mathematik und ihrer Grenzgebiete, Band 57},
   publisher={Springer-Verlag, New York-Heidelberg},
        date={1971},
         url={https://mathscinet.ams.org/mathscinet-getitem?mr=0445271},
}

\bib{Ar1969}{article}{
      author={Arveson, William~B.},
       title={Subalgebras of {$C^{\ast}$}-algebras},
        date={1969},
        ISSN={0001-5962},
     journal={Acta Math.},
      volume={123},
       pages={141\ndash 224},
         url={https://doi.org/10.1007/BF02392388},
}

\bib{AlfShuBook}{book}{
      author={Alfsen, Erik~M.},
      author={Shultz, Frederic~W.},
       title={State spaces of operator algebras},
      series={Mathematics: Theory \& Applications},
   publisher={Birkh\"{a}user Boston, Inc., Boston, MA},
        date={2001},
        ISBN={0-8176-3890-3},
         url={https://mathscinet.ams.org/mathscinet-getitem?mr=1828331},
        note={Basic theory, orientations, and $C^*$-products},
}

\bib{BaCaMu2016}{article}{
      author={Bannon, Jon~P.},
      author={Cameron, Jan},
      author={Mukherjee, Kunal},
       title={The modular symmetry of {M}arkov maps},
        date={2016},
        ISSN={0022-247X},
     journal={J. Math. Anal. Appl.},
      volume={439},
       pages={701\ndash 708},
         url={https://doi.org/10.1016/j.jmaa.2016.03.013},
}

\bib{BiDeGi2021}{article}{
      author={Bischoff, Marcel},
      author={Del~Vecchio, Simone},
      author={Giorgetti, Luca},
       title={Compact hypergroups from discrete subfactors},
        date={2021},
        ISSN={0022-1236},
     journal={J. Funct. Anal.},
      volume={281},
       pages={109004},
         url={https://mathscinet.ams.org/mathscinet-getitem?mr=4234861},
}

\bib{BiDeGi2022}{article}{
      author={Bischoff, Marcel},
      author={Del~Vecchio, Simone},
      author={Giorgetti, Luca},
       title={Galois correspondence and {F}ourier analysis on local discrete
  subfactors},
        date={2022},
        ISSN={1424-0637},
     journal={Ann. Henri Poincar\'{e}},
      volume={23},
       pages={2979\ndash 3020},
         url={https://doi.org/10.1007/s00023-022-01154-4},
}

\bib{BcEv1998-I}{article}{
      author={B{\"o}ckenhauer, Jens},
      author={Evans, David~E.},
       title={{Modular invariants, graphs and {$\alpha$}-induction for nets of
  subfactors. {I}}},
        date={1998},
        ISSN={0010-3616},
     journal={Comm. Math. Phys.},
      volume={197},
       pages={361\ndash 386},
}

\bib{BcEvKa1999}{article}{
      author={B{\"o}ckenhauer, Jens},
      author={Evans, David~E.},
      author={Kawahigashi, Yasuyuki},
       title={{On {$\alpha$}-induction, chiral generators and modular
  invariants for subfactors}},
        date={1999},
        ISSN={0010-3616},
     journal={Comm. Math. Phys.},
      volume={208},
       pages={429\ndash 487},
         url={http://dx.doi.org/10.1007/s002200050765},
}

\bib{BrGuLo1993}{article}{
      author={Brunetti, Romeo},
      author={Guido, Daniele},
      author={Longo, Roberto},
       title={{Modular structure and duality in conformal quantum field
  theory}},
        date={1993},
        ISSN={0010-3616},
     journal={Comm. Math. Phys.},
      volume={156},
       pages={201\ndash 219},
}

\bib{BlHe1995}{book}{
      author={Bloom, Walter~R.},
      author={Heyer, Herbert},
       title={Harmonic analysis of probability measures on hypergroups},
      series={de Gruyter Studies in Mathematics},
   publisher={Walter de Gruyter \& Co., Berlin},
        date={1995},
      volume={20},
        ISBN={3-11-012105-0},
         url={http://dx.doi.org/10.1515/9783110877595},
}

\bib{Bi2016}{article}{
      author={Bischoff, Marcel},
       title={Generalized orbifold construction for conformal nets},
        date={2017},
        ISSN={0129-055X},
     journal={Rev. Math. Phys.},
      volume={29},
       pages={1750002, 53},
         url={http://dx.doi.org/10.1142/S0129055X17500027},
}

\bib{Bi2019MFO}{inproceedings}{
      author={Bischoff, Marcel},
       title={Quantum operations on conformal nets},
        date={2019},
   booktitle={In {S}ubfactors and {A}pplications. {O}berwolfach {R}ep.},
      editor={Dietmar~Bisch, Vaughan F. R.~Jones, Terry J.~Gannon},
      editor={Kawahigashi, Yasuyuki},
      volume={16},
       pages={3080\ndash 3083},
        note={DOI:
  \href{https://doi.org/10.4171/OWR/2019/49}{10.4171/OWR/2019/49}},
}

\bib{BeKa1998}{book}{
      author={Berezansky, Yu.~M.},
      author={Kalyuzhnyi, A.~A.},
       title={Harmonic analysis in hypercomplex systems},
      series={Mathematics and its Applications},
   publisher={Kluwer Academic Publishers, Dordrecht},
        date={1998},
      volume={434},
        ISBN={0-7923-5029-4},
         url={https://doi.org/10.1007/978-94-017-1758-8},
}

\bib{BiKaLoRe2014-2}{book}{
      author={Bischoff, Marcel},
      author={Kawahigashi, Yasuyuki},
      author={Longo, Roberto},
      author={Rehren, Karl-Henning},
       title={Tensor categories and endomorphisms of von {N}eumann
  algebras---with applications to quantum field theory},
      series={Springer Briefs in Mathematical Physics},
   publisher={Springer, Cham},
        date={2015},
      volume={3},
        ISBN={978-3-319-14300-2; 978-3-319-14301-9},
         url={http://dx.doi.org/10.1007/978-3-319-14301-9},
}

\bib{BlaBook}{book}{
      author={Blackadar, B.},
       title={Operator algebras},
      series={Encyclopaedia of Mathematical Sciences},
   publisher={Springer-Verlag, Berlin},
        date={2006},
      volume={122},
        ISBN={978-3-540-28486-4; 3-540-28486-9},
         url={https://mathscinet.ams.org/mathscinet-getitem?mr=2188261},
        note={Theory of $C^*$-algebras and von Neumann algebras, Operator
  Algebras and Non-commutative Geometry, III},
}

\bib{BuMaTo1988}{article}{
      author={Buchholz, Detlev},
      author={Mack, G.},
      author={Todorov, Ivan},
       title={{The current algebra on the circle as a germ of local field
  theories}},
        date={1988},
     journal={Nucl. Phys., B, Proc. Suppl.},
      volume={5},
       pages={20\ndash 56},
}

\bib{BrRo1}{book}{
      author={Bratteli, Ola},
      author={Robinson, Derek~W.},
       title={Operator algebras and quantum statistical mechanics. 1},
     edition={Second},
      series={Texts and Monographs in Physics},
   publisher={Springer-Verlag, New York},
        date={1987},
        ISBN={0-387-17093-6},
         url={http://dx.doi.org/10.1007/978-3-662-02520-8},
        note={$C^\ast$- and $W^\ast$-algebras, symmetry groups, decomposition
  of states},
}

\bib{Ca2003}{article}{
      author={Carpi, Sebastiano},
       title={The {V}irasoro algebra and sectors with infinite statistical
  dimension},
        date={2003},
        ISSN={1424-0637},
     journal={Ann. Henri Poincar{\'e}},
      volume={4},
       pages={601\ndash 611},
         url={http://dx.doi.org/10.1007/s00023-003-0140-x},
}

\bib{Ca2004}{article}{
      author={Carpi, Sebastiano},
       title={{On the representation theory of {V}irasoro nets}},
        date={2004},
        ISSN={0010-3616},
     journal={Comm. Math. Phys.},
      volume={244},
       pages={261\ndash 284},
         url={http://dx.doi.org/10.1007/s00220-003-0988-0},
}

\bib{Ca1998}{article}{
      author={Carpi, Sebastiano},
       title={Absence of subsystems for the {H}aag-{K}astler net generated by
  the energy-momentum tensor in two-dimensional conformal field theory},
        date={1998},
        ISSN={0377-9017},
     journal={Lett. Math. Phys.},
      volume={45},
       pages={259\ndash 267},
         url={http://dx.doi.org/10.1023/A:1007466420114},
}

\bib{Ca1999}{article}{
      author={Carpi, Sebastiano},
       title={Classification of subsystems for the {H}aag-{K}astler nets
  generated by {$c=1$} chiral current algebras},
        date={1999},
        ISSN={0377-9017},
     journal={Lett. Math. Phys.},
      volume={47},
       pages={353\ndash 364},
         url={https://doi.org/10.1023/A:1007517131143},
}

\bib{CaCo2001}{article}{
      author={Carpi, Sebastiano},
      author={Conti, Roberto},
       title={Classification of subsystems for local nets with trivial
  superselection structure},
        date={2001},
        ISSN={0010-3616},
     journal={Comm. Math. Phys.},
      volume={217},
       pages={89\ndash 106},
         url={http://dx.doi.org/10.1007/PL00005550},
}

\bib{CaCo2001Siena}{incollection}{
      author={Carpi, Sebastiano},
      author={Conti, Roberto},
       title={Classification of subsystems, local symmetry generators and
  intrinsic definition of local observables},
        date={2001},
   booktitle={Mathematical physics in mathematics and physics ({S}iena, 2000)},
      series={Fields Inst. Commun.},
      volume={30},
   publisher={Amer. Math. Soc., Providence, RI},
       pages={83\ndash 103},
         url={https://doi.org/10.1007/pl00005550},
}

\bib{CoDe1975}{article}{
      author={Combes, Fran\c{c}ois},
      author={Delaroche, Claire},
       title={Groupe modulaire d'une esp{\'e}rance conditionnele dans une
  alg{\`e}bre de von {N}eumann},
        date={1975},
        ISSN={0037-9484},
     journal={Bull. Soc. Math. France},
      volume={103},
       pages={385\ndash 426},
         url={http://www.numdam.org/item?id=BSMF_1975__103__385_0},
}

\bib{CaGaHi2019}{article}{
      author={Carpi, Sebastiano},
      author={Gaudio, Tiziano},
      author={Hillier, Robin},
       title={Classification of unitary vertex subalgebras and conformal
  subnets for rank-one lattice chiral {CFT} models},
        date={2019},
        ISSN={0022-2488},
     journal={J. Math. Phys.},
      volume={60},
       pages={093505, 20},
         url={https://mathscinet.ams.org/mathscinet-getitem?mr=4009419},
}

\bib{Ch1974}{article}{
      author={Choi, Man~Duen},
       title={A {S}chwarz inequality for positive linear maps on
  {$C^{\ast}$}-algebras},
        date={1974},
        ISSN={0019-2082},
     journal={Illinois J. Math.},
      volume={18},
       pages={565\ndash 574},
}

\bib{CaKaLo2010}{article}{
      author={Carpi, Sebastiano},
      author={Kawahigashi, Yasuyuki},
      author={Longo, Roberto},
       title={On the {J}ones index values for conformal subnets},
        date={2010},
        ISSN={0377-9017},
     journal={Lett. Math. Phys.},
      volume={92},
       pages={99\ndash 108},
         url={https://doi.org/10.1007/s11005-010-0384-6},
}

\bib{CaKaLoWe2018}{article}{
      author={Carpi, Sebastiano},
      author={Kawahigashi, Yasuyuki},
      author={Longo, Roberto},
      author={Weiner, Mih\'{a}ly},
       title={From vertex operator algebras to conformal nets and back},
        date={2018},
        ISSN={0065-9266},
     journal={Mem. Amer. Math. Soc.},
      volume={254},
         url={https://mathscinet.ams.org/mathscinet-getitem?mr=3796433},
}

\bib{Co1973}{article}{
      author={Connes, Alain},
       title={{Une classification des facteurs de type {${I\!I\!I}$}}},
        date={1973},
     journal={Ann. Sci. {\'E}cole Norm. Sup.(4)},
      volume={6},
       pages={133\ndash 252},
}

\bib{CaWe2005}{article}{
      author={Carpi, Sebastiano},
      author={Weiner, Mih{\'a}ly},
       title={{On the uniqueness of diffeomorphism symmetry in conformal field
  theory}},
        date={2005},
        ISSN={0010-3616},
     journal={Comm. Math. Phys.},
      volume={258},
       pages={203\ndash 221},
         url={http://dx.doi.org/10.1007/s00220-005-1335-4},
}

\bib{DoHaRo1969I}{article}{
      author={Doplicher, Sergio},
      author={Haag, Rudolf},
      author={Roberts, John~E.},
       title={{Fields, observables and gauge transformations. I}},
        date={1969},
        ISSN={0010-3616},
     journal={Comm. Math. Phys.},
      volume={13},
       pages={1\ndash 23},
}

\bib{DoHaRo1971}{article}{
      author={Doplicher, Sergio},
      author={Haag, Rudolf},
      author={Roberts, John~E.},
       title={Local observables and particle statistics. {I}},
        date={1971},
        ISSN={0010-3616},
     journal={Comm. Math. Phys.},
      volume={23},
       pages={199\ndash 230},
}

\bib{DoHaRo1974}{article}{
      author={Doplicher, Sergio},
      author={Haag, Rudolf},
      author={Roberts, John~E.},
       title={Local observables and particle statistics. {II}},
        date={1974},
        ISSN={0010-3616},
     journal={Comm. Math. Phys.},
      volume={35},
       pages={49\ndash 85},
}

\bib{DoLo1984}{article}{
      author={Doplicher, S.},
      author={Longo, R.},
       title={{Standard and split inclusions of von {N}eumann algebras}},
        date={1984},
        ISSN={0020-9910},
     journal={Invent. Math.},
      volume={75},
       pages={493\ndash 536},
         url={http://dx.doi.org/10.1007/BF01388641},
}

\bib{DLR2001}{article}{
      author={D'Antoni, Claudio},
      author={Longo, Roberto},
      author={Radulescu, F.},
       title={{Conformal nets, maximal temperature and models from free
  probability}},
        date={2001},
     journal={J. Oper. Theory},
      volume={45},
       pages={195\ndash 208},
}

\bib{DoRo1990}{article}{
      author={Doplicher, Sergio},
      author={Roberts, John~E.},
       title={Why there is a field algebra with a compact gauge group
  describing the superselection structure in particle physics},
        date={1990},
        ISSN={0010-3616},
     journal={Comm. Math. Phys.},
      volume={131},
       pages={51\ndash 107},
         url={http://projecteuclid.org/euclid.cmp/1104200703},
}

\bib{DeGi2018}{article}{
      author={Del~Vecchio, Simone},
      author={Giorgetti, Luca},
       title={Infinite index extensions of local nets and defects},
        date={2018},
        ISSN={0129-055X},
     journal={Rev. Math. Phys.},
      volume={30},
       pages={1850002, 58},
         url={https://doi.org/10.1142/S0129055X18500022},
}

\bib{DVIT20}{article}{
      author={Del~Vecchio, Simone},
      author={Iovieno, Stefano},
      author={Tanimoto, Yoh},
       title={Solitons and nonsmooth diffeomorphisms in conformal nets},
        date={2020},
        ISSN={0010-3616},
     journal={Comm. Math. Phys.},
      volume={375},
       pages={391\ndash 427},
         url={https://mathscinet.ams.org/mathscinet-getitem?mr=4082169},
}

\bib{CDIT2021}{article}{
      author={Carpi, Sebastiano},
      author={Del~Vecchio, Simone},
      author={Iovieno, Stefano},
      author={Tanimoto, Yoh},
       title={Positive energy representations of {S}obolev diffeomorphism
  groups of the circle},
        date={2021},
        ISSN={1664-2368},
     journal={Anal. Math. Phys.},
      volume={11},
       pages={Paper No. 12, 36},
         url={https://doi.org/10.1007/s13324-020-00429-5},
}

\bib{EGNO15}{book}{
      author={Etingof, Pavel},
      author={Gelaki, Shlomo},
      author={Nikshych, Dmitri},
      author={Ostrik, Victor},
       title={Tensor categories},
      series={Mathematical Surveys and Monographs},
   publisher={American Mathematical Society, Providence, RI},
        date={2015},
      volume={205},
        ISBN={978-1-4704-2024-6},
}

\bib{FeHo2005}{article}{
      author={Fewster, Christopher~J.},
      author={Hollands, Stefan},
       title={Quantum energy inequalities in two-dimensional conformal field
  theory},
        date={2005},
        ISSN={0129-055X},
     journal={Rev. Math. Phys.},
      volume={17},
       pages={577\ndash 612},
         url={https://doi.org/10.1142/S0129055X05002406},
}

\bib{FrJr1996}{article}{
      author={Fredenhagen, K.},
      author={J{\"o}r{\ss}, Martin},
       title={{Conformal Haag-Kastler nets, pointlike localized fields and the
  existence of operator product expansions}},
        date={1996},
     journal={Comm. Math. Phys.},
      volume={176},
       pages={541\ndash 554},
}

\bib{FrQiSh1985}{incollection}{
      author={Friedan, Daniel},
      author={Qiu, Zongan},
      author={Shenker, Stephen},
       title={Conformal invariance, unitarity and two-dimensional critical
  exponents},
        date={1985},
   booktitle={Vertex operators in mathematics and physics ({B}erkeley,
  {C}alif., 1983)},
      series={Math. Sci. Res. Inst. Publ.},
      volume={3},
   publisher={Springer, New York},
       pages={419\ndash 449},
         url={https://mathscinet.ams.org/mathscinet-getitem?mr=781390},
}

\bib{Fr1993}{incollection}{
      author={Fredenhagen, Klaus},
       title={Superselection sectors with infinite statistical dimension},
        date={1994},
   booktitle={Subfactors ({K}yuzeso, 1993)},
   publisher={World Sci. Publ., River Edge, NJ},
       pages={242\ndash 258},
}

\bib{FrReSc1989}{article}{
      author={Fredenhagen, K.},
      author={Rehren, K.-H.},
      author={Schroer, B.},
       title={{Superselection sectors with braid group statistics and exchange
  algebras. {I}.\ {G}eneral theory}},
        date={1989},
        ISSN={0010-3616},
     journal={Comm. Math. Phys.},
      volume={125},
       pages={201\ndash 226},
         url={http://projecteuclid.org/getRecord?id=euclid.cmp/1104179464},
}

\bib{FrReSc1992}{article}{
      author={Fredenhagen, Klaus},
      author={Rehren, Karl-Henning},
      author={Schroer, Bert},
       title={Superselection sectors with braid group statistics and exchange
  algebras. {II}. {G}eometric aspects and conformal covariance},
        date={1992},
        ISSN={0129-055X},
     journal={Rev. Math. Phys.},
      pages={113\ndash 157},
         url={http://dx.doi.org/10.1142/S0129055X92000170},
        note={SI1 (Special issue)}, 
}

\bib{GaFr1993}{article}{
      author={Gabbiani, Fabrizio},
      author={Fr{\"o}hlich, J{\"u}rg},
       title={{Operator algebras and conformal field theory}},
        date={1993},
        ISSN={0010-3616},
     journal={Comm. Math. Phys.},
      volume={155},
       pages={569\ndash 640},
}

\bib{Gi2022}{article}{
      author={Giorgetti, Luca},
       title={A planar algebraic description of conditional expectations},
        date={2022},
        ISSN={0129-167X},
     journal={Internat. J. Math.},
      volume={33},
       pages={Paper No. 2250037, 23},
         url={https://doi.org/10.1142/S0129167X22500379},
}

\bib{GoKeOl1986}{article}{
      author={Goddard, P.},
      author={Kent, A.},
      author={Olive, D.},
       title={{Unitary representations of the {V}irasoro and super-{V}irasoro
  algebras}},
        date={1986},
        ISSN={0010-3616},
     journal={Comm. Math. Phys.},
      volume={103},
       pages={105\ndash 119},
         url={http://projecteuclid.org/getRecord?id=euclid.cmp/1104114626},
}

\bib{GiLo2019}{article}{
      author={Giorgetti, Luca},
      author={Longo, Roberto},
       title={Minimal index and dimension for 2-{$C^*$}-categories with
  finite-dimensional centers},
        date={2019},
        ISSN={0010-3616},
     journal={Comm. Math. Phys.},
      volume={370},
       pages={719\ndash 757},
         url={https://doi.org/10.1007/s00220-018-3266-x},
}

\bib{GuLo1996}{article}{
      author={Guido, Daniele},
      author={Longo, Roberto},
       title={The conformal spin and statistics theorem},
        date={1996},
        ISSN={0010-3616},
     journal={Comm. Math. Phys.},
      volume={181},
       pages={11\ndash 35},
         url={http://projecteuclid.org/euclid.cmp/1104287623},
}

\bib{GuLoWi1998}{article}{
      author={Guido, Daniele},
      author={Longo, Roberto},
      author={Wiesbrock, Hans-Werner},
       title={{Extensions of Conformal Nets and Superselection Structures}},
        date={1998},
     journal={Comm. Math. Phys.},
      volume={192},
       pages={217\ndash 244},
}


\bib{GPRR23}{article}{
      author={Giorgetti, Luca},
      author={Parzygnat, Arthur~J.},
      author={Ranallo, Alessio},
      author={Russo, Benjamin~P.},
      title={{B}ayesian inversion and the {T}omita--{T}akesaki modular group},
      date={2023},
      journal = {The Quarterly Journal of Mathematics},
      note={\href{https://doi.org/10.1093/qmath/haad014}{haad014, https://doi.org/10.1093/qmath/haad014}},
}

\bib{GiRe2018}{article}{
      author={Giorgetti, Luca},
      author={Rehren, Karl-Henning},
       title={Braided categories of endomorphisms as invariants for local
  quantum field theories},
        date={2018},
        ISSN={0010-3616},
     journal={Comm. Math. Phys.},
      volume={357},
       pages={3\ndash 41},
         url={https://doi.org/10.1007/s00220-017-2937-3},
}

\bib{GiYu2019}{article}{
      author={Giorgetti, Luca},
      author={Yuan, Wei},
       title={Realization of rigid {$C^\ast$}-tensor categories via {T}omita
  bimodules},
        date={2019},
        ISSN={0379-4024},
     journal={J. Operator Theory},
      volume={81},
       pages={433\ndash 479},
}

\bib{Ha}{book}{
      author={Haag, Rudolf},
       title={{Local quantum physics}},
   publisher={Springer Berlin},
        date={1996},
}

\bib{IzLoPo1998}{article}{
      author={Izumi, Masaki},
      author={Longo, Roberto},
      author={Popa, Sorin},
       title={{A {G}alois correspondence for compact groups of automorphisms of
  von {N}eumann algebras with a generalization to {K}ac algebras}},
        date={1998},
        ISSN={0022-1236},
     journal={J. Funct. Anal.},
      volume={155},
       pages={25\ndash 63},
         url={http://dx.doi.org/10.1006/jfan.1997.3228},
}

\bib{Jo1983}{article}{
      author={Jones, V. F.~R.},
       title={{Index for subfactors}},
        date={1983},
        ISSN={0020-9910},
     journal={Invent. Math.},
      volume={72},
       pages={1\ndash 25},
         url={http://dx.doi.org/10.1007/BF01389127},
}

\bib{Koe2004}{article}{
      author={K\"{o}ster, S.},
       title={Local nature of coset models},
        date={2004},
        ISSN={0129-055X},
     journal={Rev. Math. Phys.},
      volume={16},
       pages={353\ndash 382},
         url={https://doi.org/10.1142/S0129055X0400200X},
}

\bib{Ka1952}{article}{
      author={Kadison, Richard~V.},
       title={A generalized {S}chwarz inequality and algebraic invariants for
  operator algebras},
        date={1952},
        ISSN={0003-486X},
     journal={Ann. of Math. (2)},
      volume={56},
       pages={494\ndash 503},
}

\bib{KaLo2004}{article}{
      author={Kawahigashi, Y.},
      author={Longo, Roberto},
       title={{Classification of local conformal nets. Case {$c < 1$}.}},
        date={2004},
        ISSN={0003-486X},
     journal={Ann. Math.},
      volume={160},
       pages={493\ndash 522},
}

\bib{KaLoMg2001}{article}{
      author={Kawahigashi, Y.},
      author={Longo, Roberto},
      author={M{\"u}ger, Michael},
       title={{Multi-Interval Subfactors and Modularity of Representations in
  Conformal Field Theory}},
        date={2001},
     journal={Comm. Math. Phys.},
      volume={219},
       pages={631\ndash 669},
}

\bib{KaRaBombay}{book}{
      author={Kac, V.~G.},
      author={Raina, A.~K.},
       title={Bombay lectures on highest weight representations of
  infinite-dimensional {L}ie algebras},
      series={Advanced Series in Mathematical Physics},
   publisher={World Scientific Publishing Co., Inc., Teaneck, NJ},
        date={1987},
      volume={2},
        ISBN={9971-50-395-6; 9971-50-396-4},
         url={https://mathscinet.ams.org/mathscinet-getitem?mr=1021978},
}

\bib{Lo2003}{article}{
      author={Longo, Roberto},
       title={{Conformal Subnets and Intermediate Subfactors}},
        date={2003},
        ISSN={0010-3616},
     journal={Comm. Math. Phys.},
      volume={237},
       pages={7\ndash 30},
         url={http://dx.doi.org/10.1007/s00220-003-0814-8},
}

\bib{Lo2}{article}{
      author={Longo, Roberto},
       title={{Lecture Notes on Conformal Nets}},
        date={2008},
      eprint={https://www.mat.uniroma2.it/longo/lecture-notes.html},
         url={https://www.mat.uniroma2.it/longo/lecture-notes.html},
        note={first part published in
  {\protect\NoHyper\cite{Lon08}\protect\endNoHyper}},
}

\bib{Lon08}{incollection}{
      author={Longo, Roberto},
       title={{Real {H}ilbert subspaces, modular theory, {${\rm SL}(2,\Bbb R)$}
  and {CFT}}},
        date={2008},
   booktitle={{Von {N}eumann algebras in {S}ibiu}},
      series={{Theta Ser. Adv. Math.}},
      volume={10},
   publisher={Theta, Bucharest},
       pages={33\ndash 91},
}

\bib{Lo1989}{article}{
      author={Longo, Roberto},
       title={{Index of subfactors and statistics of quantum fields. I}},
        date={1989},
     journal={Comm. Math. Phys.},
      volume={126},
       pages={217\ndash 247},
}

\bib{Lo1994}{article}{
      author={Longo, Roberto},
       title={{A duality for {H}opf algebras and for subfactors. {I}}},
        date={1994},
        ISSN={0010-3616},
     journal={Comm. Math. Phys.},
      volume={159},
       pages={133\ndash 150},
         url={http://projecteuclid.org/getRecord?id=euclid.cmp/1104254494},
}

\bib{Lon97}{article}{
      author={Longo, Roberto},
       title={An analogue of the {K}ac-{W}akimoto formula and black hole
  conditional entropy},
        date={1997},
        ISSN={0010-3616},
     journal={Comm. Math. Phys.},
      volume={186},
       pages={451\ndash 479},
         url={http://dx.doi.org/10.1007/s002200050116},
}

\bib{LoRe1995}{article}{
      author={Longo, Roberto},
      author={Rehren, Karl-Henning},
       title={{Nets of subfactors}},
        date={1995},
     journal={Rev. Math. Phys.},
      volume={7},
       pages={567\ndash 597},
}

\bib{LoRo1997}{article}{
      author={Longo, R.},
      author={Roberts, J.~E.},
       title={{A theory of dimension}},
        date={1997},
        ISSN={0920-3036},
     journal={K-Theory},
      volume={11},
       pages={103\ndash 159},
         url={http://dx.doi.org/10.1023/A:1007714415067},
}

\bib{Mi1984}{incollection}{
      author={Milnor, J.},
       title={Remarks on infinite-dimensional {L}ie groups},
        date={1984},
   booktitle={Relativity, groups and topology, {II} ({L}es {H}ouches, 1983)},
   publisher={North-Holland, Amsterdam},
       pages={1007\ndash 1057},
         url={https://mathscinet.ams.org/mathscinet-getitem?mr=830252},
}

\bib{MoTaWe2018}{article}{
      author={Morinelli, Vincenzo},
      author={Tanimoto, Yoh},
      author={Weiner, Mih\'{a}ly},
       title={Conformal covariance and the split property},
        date={2018},
        ISSN={0010-3616},
     journal={Comm. Math. Phys.},
      volume={357},
       pages={379\ndash 406},
         url={https://doi.org/10.1007/s00220-017-2961-3},
}

\bib{Mue10tour}{article}{
      author={M{\"u}ger, Michael},
       title={Tensor categories: a selective guided tour},
        date={2010},
        ISSN={0041-6932},
     journal={Rev. Un. Mat. Argentina},
      volume={51},
      number={1},
       pages={95\ndash 163},
}

\bib{NiStZs2003}{article}{
      author={Niculescu, Constantin~P.},
      author={Str{\"o}h, Anton},
      author={Zsid{\'o}, L{\'a}szl{\'o}},
       title={Noncommutative extensions of classical and multiple recurrence
  theorems},
        date={2003},
        ISSN={0379-4024},
     journal={J. Operator Theory},
      volume={50},
       pages={3\ndash 52},
}

\bib{OhPe1993}{book}{
      author={Ohya, Masanori},
      author={Petz, D\'enes},
       title={Quantum entropy and its use},
      series={Texts and Monographs in Physics},
   publisher={Springer-Verlag, Berlin},
        date={1993},
        ISBN={3-540-54881-5},
         url={https://doi.org/10.1007/978-3-642-57997-4},
}

\bib{Pa2002}{book}{
      author={Paulsen, Vern},
       title={Completely bounded maps and operator algebras},
      series={Cambridge Studies in Advanced Mathematics},
   publisher={Cambridge University Press, Cambridge},
        date={2002},
      volume={78},
        ISBN={0-521-81669-6},
         url={https://mathscinet.ams.org/mathscinet-getitem?mr=1976867},
}

\bib{Pe1984}{article}{
      author={Petz, D\'{e}nes},
       title={A dual in von {N}eumann algebras with weights},
        date={1984},
        ISSN={0033-5606},
     journal={Quart. J. Math. Oxford Ser. (2)},
      volume={35},
       pages={475\ndash 483},
         url={https://doi.org/10.1093/qmath/35.4.475},
}

\bib{Pet88}{article}{
      author={Petz, D{\'e}nes},
       title={Sufficiency of channels over von~{N}eumann algebras},
        date={1988},
     journal={Q. J. Math.},
      volume={39},
       pages={97\ndash 108},
}

\bib{Phe01}{book}{
      author={Phelps, Robert~R.},
       title={Lectures on {C}hoquet's theorem},
     edition={Second},
      series={Lecture Notes in Mathematics},
   publisher={Springer-Verlag, Berlin},
        date={2001},
      volume={1757},
        ISBN={3-540-41834-2},
         url={https://mathscinet.ams.org/mathscinet-getitem?mr=1835574},
}

\bib{PaRu22}{article}{
      author={Parzygnat, Arthur~J.},
      author={Russo, Benjamin~P.},
       title={A non-commutative {B}ayes' theorem},
        date={2022},
        ISSN={0024-3795},
     journal={Linear Algebra Appl.},
      volume={644},
       pages={28\ndash 94},
         url={https://doi.org/10.1016/j.laa.2022.02.030},
}

\bib{Re1994}{article}{
      author={Rehren, Karl-Henning},
       title={{A new view of the {V}irasoro algebra}},
        date={1994},
        ISSN={0377-9017},
     journal={Lett. Math. Phys.},
      volume={30},
       pages={125\ndash 130},
         url={http://dx.doi.org/10.1007/BF00939700},
}

\bib{St1997}{incollection}{
      author={St{\o}rmer, E.},
       title={Conditional expectations and projection maps of von {N}eumann
  algebras},
        date={1997},
   booktitle={Operator algebras and applications ({S}amos, 1996)},
      series={NATO Adv. Sci. Inst. Ser. C Math. Phys. Sci.},
      volume={495},
   publisher={Kluwer Acad. Publ., Dordrecht},
       pages={449\ndash 461},
         url={https://mathscinet.ams.org/mathscinet-getitem?mr=1462691},
}

\bib{SuWi2003}{article}{
      author={Sunder, V.~S.},
      author={Wildberger, N.~J.},
       title={Actions of finite hypergroups},
        date={2003},
        ISSN={0925-9899},
     journal={J. Algebraic Combin.},
      volume={18},
       pages={135\ndash 151},
         url={http://dx.doi.org/10.1023/A:1025107014451},
}

\bib{Ta1}{book}{
      author={Takesaki, M.},
       title={{Theory of operator algebras. {I}}},
      series={{Encyclopaedia of Mathematical Sciences}},
   publisher={Springer-Verlag},
     address={Berlin},
        date={2002},
      volume={124},
        ISBN={3-540-42248-X},
        note={Reprint of the first (1979) edition, Operator Algebras and
  Non-commutative Geometry, 5},
}

\bib{Ta1972}{article}{
      author={Takesaki, Masamichi},
       title={Conditional expectations in von {N}eumann algebras},
        date={1972},
     journal={J. Functional Analysis},
      volume={9},
       pages={306\ndash 321},
}

\bib{Vr1979}{article}{
      author={Vrem, Richard~C.},
       title={Harmonic analysis on compact hypergroups},
        date={1979},
        ISSN={0030-8730},
     journal={Pacific J. Math.},
      volume={85},
       pages={239\ndash 251},
         url={http://projecteuclid.org/euclid.pjm/1102784093},
}

\bib{WeiPhD}{thesis}{
      author={Weiner, Mih{\'a}ly},
       title={{Conformal covariance and related properties of chiral QFT}},
        type={Ph.D. Thesis},
        date={2005},
        note={\href{https://arxiv.org/abs/math/0703336}{Available at arXiv:0703336}}
}

\bib{WildeBook}{book}{
      author={Wilde, Mark~M.},
       title={Quantum information theory},
     edition={Second},
   publisher={Cambridge University Press, Cambridge},
        date={2017},
        ISBN={978-1-107-17616-4},
         url={https://doi.org/10.1017/9781316809976},
}

\bib{Xu2000orbi}{article}{
      author={Xu, Feng},
       title={{Algebraic orbifold conformal field theories}},
        date={2000},
     journal={Proc. Nat. Acad. Sci. U.S.A.},
      volume={97},
       pages={14069},
}

\bib{Xu2001siena}{incollection}{
      author={Xu, Feng},
       title={Algebraic orbifold conformal field theories},
        date={2001},
   booktitle={Mathematical physics in mathematics and physics ({S}iena, 2000)},
      series={Fields Inst. Commun.},
      volume={30},
   publisher={Amer. Math. Soc., Providence, RI},
       pages={429\ndash 448},
}

\bib{Xu2005}{article}{
      author={Xu, Feng},
       title={{Strong additivity and conformal nets}},
        date={2005},
        ISSN={0030-8730},
     journal={Pacific J. Math.},
      volume={221},
       pages={167\ndash 199},
         url={http://dx.doi.org/10.2140/pjm.2005.221.167},
}

\bib{To1959}{article}{
      author={Tomiyama, Jun},
       title={On the projection of norm one in {$W\sp*$}-algebras. {III}},
        date={1959},
        ISSN={0040-8735},
     journal={T\^ohoku Math. J.},
      volume={11},
       pages={125\ndash 129},
}

\end{biblist}
\end{bibdiv}
\end{document}